\documentclass[amssymb,amscd,11pt,20pt]{amsart}
\usepackage{amsmath,amsthm}
\usepackage{amsfonts}
\usepackage{amscd}
\usepackage[english]{babel}
\usepackage{amscd}
\usepackage[all]{xy}

\usepackage[all]{xy}
\usepackage{amsmath,latexsym,amssymb,verbatim}
\input arrow.tex

\newcommand{\codim}{{\operatorname{codim}}}

\newcommand{\LL}{{\mathbb L}}
\newcommand{\FF}{{\mathtt F}}

\newcommand{\Lc}{{\mathcal L}}

\newcommand{\Sa}{{\mathbb S}}
\newcommand{\ZZ}{{\mathbb Z}}

\newcommand{\PP}{\mathbb P}
\newcommand{\AAA}{\mathbb A}
\newcommand{\RR}{\mathbb R}
\newcommand{\0}{\mathbf 0}
\newcommand{\un}{\mathbf 1}

\newcommand{\wh}{\widehat}

\newcommand{\enumera}{\begin{enumerate}}
\newcommand{\eenumera}{\end{enumerate}}
\newcommand{\C}{{\mathcal C}}

\newcommand{\A}{{\mathbb A}}

\DeclareMathOperator{\Hom}{{Hom}}

\newcommand{\alineas}[1]{\begin{array}{#1}}
\newcommand{\alinea}{\begin{array}{l}}
\newcommand{\ealinea}{\end{array}}
\newcommand{\ealineas}{\end{array}}

\newcommand{\pun}{{\scriptscriptstyle \bullet}}

\newcommand{\HHom}{\mathcal{H}om}

\newcommand{\supp}{\operatorname{supp}}

\theoremstyle{plain}
\newtheorem{thm}{Theorem}[section]
\newtheorem{lem}[thm]{Lemma}
\newtheorem{cor}[thm]{Corollary}
\newtheorem{prop}[thm]{Proposition}
\newtheorem{defn}[thm]{Definition}
\newtheorem{rem}[thm]{Remark}

\newtheorem{ejems}[thm]{Examples}

\newtheorem*{ex}{Example}
\newtheorem*{exs}{Examples}

\numberwithin{equation}{thm}

\begin{document}

\title{Dualizing and canonical complexes on finite posets II: properness}
%\author{Amelia \'{A}lvarez, Fernando Sancho, and Pedro Sancho}
%\date{10-1-2005}

\author{Fernando Sancho de Salas}

\address{ Departamento de
Matem\'aticas and Instituto Universitario de F\'isica Fundamental y Matem\'aticas (IUFFyM)\newline
Universidad de Salamanca\newline  Plaza de la Merced 1-4\\
37008 Salamanca\newline  Spain}
\email{fsancho@usal.es}



\thanks {Work supported by Grant PID2021-128665NB-I00 funded by MCIN/AEI/ 10.13039/501100011033 and, as appropriate, by ``ERDF A way of making Europe''.}
 
\begin{abstract}   We develop a theory of proper spaces and maps in the context of finite posets.
\end{abstract}

 \maketitle

\section*{Introduction} 

Grothendieck's theory of dualizing and canonical complexes, along with their relationship, was developed in \cite{ST2} within the context of finite posets. This framework was utilized there to establish a subsequent theory of Cohen--Macaulayness that mimics Grothendieck's ideas. In the present paper, we employ the theory of canonical and dualizing complexes to develop a theory of proper maps and spaces on finite posets.

In classical geometric contexts (that is, that of locally compact Hausdorff spaces or schemes) proper spaces and proper maps play a relevant role and projective spaces or maps are a fundamental example of them. For instance, properness plays an important role in Verdier-Grothendieck duality. In this paper we investigate what properness should be in the context of finite posets (i.e., finite and $T_0$ topological spaces, which are called finite spaces here after). In this context, projective and affine spaces exist, denoted by $\PP^n_{\FF_1}$ and $\A^n_{\FF_1}$ in this paper  because they are the underlying topological spaces of Deitmar's projective and affine spaces over the field with one element (\cite{Thas}, \cite{Deitmar}).  They are the most basic examples of finite posets: $\PP^n_{\FF_1}$ is the poset of all (non-empty) subsets of a set with $n+1$ elements, and $\A^n_{\FF_1}$ is the poset of all  subsets of a set with $n$ elements. They serve as the ambient spaces of (finite) abstract simplicial complexes (the choice between $\A^n_{\FF_1}$ and $\PP^n_{\FF_1}$   depends on whether  or not  one admits the empty subset as a simplex). One would desire $\PP^n_{\FF_1}$ to be proper but not $\A^n_{\FF_1}$ (for $n>0$), and a closed subset of $\PP^n_{\FF_1}$ to be proper but not a closed subset of $\A^n_{\FF_1}$ (unless it is a point). After studying dualizing and canonical complexes on finite posets, we noticed that if $X$ is  a closed subset of  $\PP^n_{\FF_1}$, then its global dualizing complex $D_X$ is also a canonical complex, which fails to occur if $X$ is a closed subset of $\A^n_{\FF_1}$ (unless $X$ is a point). This parallels the classic geometric context: if $X$ is proper, the dualizing complex is a canonical complex. Thus, the starting idea is simple: take this property as the definition of a proper space, at least for a dualizable space $X$ (a space admitting a canonical complex).  For the general—non-dualizable—case, we need to introduce some concepts and notations. 

Let $X$ be a finite space. We say that $X$ is {\emph {local}} if it has a unique closed point, denoted by $\0$. For each $p\in X$, $U_p$ denotes the smallest open subset of $X$ containing $p$. Consequentially, $U_p$ is local, with closed point $p$. Let $D_X$ be the global dualizing complex on $X$, i.e., the complex of sheaves on $X$ that dualizes cohomology (Definition \ref{defglobaldualizing}). If $X$ is local, let $D_X^\0$ be the local dualizing complex of $X$, which dualizes local cohomology at $\0$ (Definition \ref{deflocaldualizing}). We say that $X$ is \emph{dualizable} if it admits a canonical complex (Definition \ref{defcanonical}) and $X$ is \emph{locally dualizable} if $U_p$ is dualizable for every $p\in X$. Locally dualizable spaces are characterized as those spaces $X$ satisfying: (1) $X$ is catenary (every closed interval is pure), and (2) every open interval is a cohomological sphere (Definition \ref{coh-sphere}). Furthermore, a local space $X$ is dualizable if and only if the local dualizing complex $D_X^\0$ is a canonical complex. The category of locally dualizable spaces contains all simplicial and locally simplicial spaces, as well as other posets like spheres and simplicial posets; moreover,  it is closed under taking open or closed subsets, and products.

Let us define now a proper space. We shall prove   that for every finite space $X$ and every closed point $x_0$, there exists a natural morphism
\[\phi_{x_0}\colon (D_X)_{\vert U_{x_0}}\to D_{U_{x_0}}^{x_0},\] which motivates the following definition.

\smallskip
\noindent{\bf\ \ Definition.} $X$ is proper if $\phi_{x_0}$ is an isomorphism for every closed point $x_0$; that is, if the global and local dualizing complexes are compatible.
\smallskip

If $X$ is locally dualizable, then $X$ is proper if and only if $D_X$ is a canonical complex (Corollary \ref{dualizante=canonico}), as mentioned above.

Naturally, this definition of properness requires a more manageable criterion. For each $p\in X$, let $C_p$ be the closure of $\{p\}$ and set $\partial C_p=C_p- \{p\}$. The criterion is provided by the following result (see Theorems \ref{thmproperspace} and  \ref{proper+dualizable}):

\smallskip
\noindent{\bf\ \ Theorem.} {\it Let $X$ be a finite space. The following conditions are equivalent:
\begin{enumerate}\item $X$ is proper.
\item  For every closed point $x_0\in X$ and every $p>x_0 $, the space $\partial C_p-\{x_0\}$ is cohomologically trivial (Definition \ref{c-trivial}). 
\item For every closed point $x_0\in X$ and every $p>x_0 $, one has
\[ H^i_{x_0}(\partial C_p,\ZZ) = H^i_{\text{\rm red}}(\partial C_p,\ZZ)\quad \text{for all }i,\]
where $H^i_{x_0}$ denotes the local cohomology groups at $x_0$ and $H^i_{\text{\rm red}}$ denotes the reduced cohomology groups.
\end{enumerate}

Moreover, if $X$ is locally dualizable, the following conditions are equivalent:

\begin{enumerate}
\item $X$ is proper.
\item For every $p\in X$,    $\partial C_p$ is pure and  a cohomological sphere.
\item For every $p\in X$, $\partial C_p $ is pure, and $\partial C_p-\{x_0\}$ is cohomologically trivial for some closed point $x_0\in\partial C_p$.
\end{enumerate}}
\smallskip

We now turn our attention to proper maps, formulated in such a way that a finite space $X$ is proper if and only if the projection to a point, $X\to \{*\}$, is proper. A first approximation to proper maps was given in \cite{Sanchoetal} through the notion of a cohomologically proper map (or c-proper map for short). These are maps whose cohomological behavior mimics that of proper maps between locally compact Hausdorff spaces. More precisely, a proper map $f\colon X\to Y$ between locally compact Hausdorff spaces satisfies the following cohomological properties (see \cite{Iversen}):

(1) Base change for cohomology: for every sheaf $F$ on $X$ and every $y\in Y$, there is a natural isomorphism
\[(R^if_*F)_y= H^i(f^{-1}(y),F_{\vert f^{-1}(y)}),\quad \text{for all }  i.\]

(2) Duality is local on $Y$: the functor $f^\times\colon \mathrm{D}^+(Y)\to \mathrm{D}^+(X) $ (the right adjoint of $\RR f_*$) is local on $Y$. Equivalently, the local form of duality holds:  for every $F\in \mathrm{D}^+(X),K\in \mathrm{D}^+(Y)$, there is a natural   isomorphism
\[\RR f_*\RR\HHom_X^\pun(F,f^\times K)\overset\sim\to\RR\HHom_Y^\pun(\RR f_*F,K).\] 
Now, a continuous map $\colon X\to Y$ between finite spaces is c-proper if the induced map $f_{\vert C_p}\colon C_p\to C_{f(p)}$ has cohomologically trivial fibres, for every $p\in X$. It is proved in \cite{Sanchoetal} that $f$ is c-proper if and only if $f$ satisfies base change for cohomology, which in turn is equivalent to duality being local on $Y$.

Thus, it is natural to require a proper map to be c-proper. But being c-proper is too  weak a condition because the projection to a point $X\to\{*\}$ is   c-proper for every space $X$. We must therefore impose an additional condition on $f$, which serves as a relative version of the properness condition on $X$ introduced above (the compatibility of global and local dualizing complexes). This is achieved as follows.

We say that a point $x_0\in X$ is $f$-closed if $x_0$ is closed in $f^{-1}(f(x_0))$. We shall prove (Proposition \ref{f-localdual}) that the functor $f^\times$ (for $f\colon f^{-1}(U_{f(x_0})\to U_{f(x_0)}$) relates the local dualizing complexes of $U_{x_0}$ and $U_{f(x_0)}$ via a natural morphism
\[ \phi_{x_0}\colon (f^\times D_{U_{f(x_0)}}^{f(x_0)})_{\vert U_{x_0}}\to D_{U_{x_0}}^{x_0}.\] We say that $f^\times$ is \emph{compatible with local dualizing complexes} if $\phi_{x_0}$ is an isomorphism for every $f$-closed point $x_0$. We  then define a proper map $f\colon X\to Y$ as a c-proper map such that $f^\times$ is compatible with local dualizing complexes. Consequently, $X$ is proper if and only if $X\to\{*\}$ is proper. 

We shall prove that the fibers of a proper map are proper, although the converse does not hold (paralleling the classical geometric contexts). The extra condition (besides c-properness) involves a formula relating the reduced cohomology of the intervals $(x_0,p)$ (with $x_0$ an $f$-closed point), $(f(x_0),f(p))$ and the generic fibre  $G_p$     of $f_{\vert C_p}\colon C_p\to C_{f(p)}$. See Theorem \ref{thmproperonfibres} for the precise statement. Under local dualizability conditions on $X$ and $Y$ the result reads as follows (see Theorem \ref{properfibresld}):

\smallskip
\noindent{\bf \ \ Theorem.} {\it  Let $f\colon X\to Y$ be a c-proper map between locally dualizable spaces. The following conditions are equivalent:
\begin{enumerate}\item $f$ is proper.
\item $f$ has proper fibres and, for every $f$-closed point $x_0$ and every $p>x_0$ such that $f(p)> f(x_0)$, one has: 
\[\dim [x_0,p]=\dim [f(x_0),f(p)] + \dim G_p, 
\] where $G_p$  is the closure of $\{p\}$ in $f^{-1}(f(p))$, i.e.,  the generic fibre of $f_{\vert C_p}\colon C_p\to C_{f(p)}$.
\end{enumerate}}
\smallskip

As a consequence, we shall show (Corollary \ref{properonvarieties}) that a proper map $f\colon X\to Y$ between locally dualizable spaces satisfies the dimensional formula
\[ \dim C_p=\dim C_{f(p)}+\dim G_p\] for every $p\in X$. Conversely, every c-proper map with proper fibres and satisfying this formula is proper under some extra purity conditions on $X$ and $Y$ (see Proposition \ref{proper-formula}). Notice the analogy of this identity with the dimensional formula for a dominant map between algebraic varieties.

We now outline the organization of the paper and summarize our other results.  Section \ref{prelim-section} is  devoted to recall some basic facts on finite spaces and their derived categories, local and global dualizing complexes on them, canonical complexes and locally dualizable spaces, Cohen--Macaulayness and cohomologically proper maps. Only a few results in this section  are new: Proposition \ref{transitivity}, Corollaries \ref{loc-dual-op} and \ref{Gor}, and Proposition \ref{cohpropsup} (the latter of which plays a crucial role in the study of proper maps).

Section \ref{loc-glob-compatibility} investigates the compatibility between local and global dualizing complexes. Specifically,  we  construct the previously mentioned morphism $\phi_{x_0}\colon (f^\times D_{U_{f(x_0)}}^{f(x_0)})_{\vert U_{x_0}}\to D_{U_{x_0}}^{x_0}$ and provide various  characterizations for this morphism to be an isomorphism in Proposition \ref{equivalencias}, which serves as the key technical result for  Theorems \ref{thmproperonfibres} and \ref{properfibresld} mentioned above.

Section \ref{section-properness} is dedicated to proper maps and proper spaces. In addition to the results mentioned above and several expected general properties (e.g., that the composition of proper maps is proper and that closed immersions are proper), we establish further results under the locally dualizable hypothesis in subsection \ref{subsection-ld-properness}.  For instance, we prove that proper maps are precisely those c-proper maps such that $f^\times $ maps canonical complexes to canonical complexes (Proposition \ref{proper-canonical}). This constitutes the relative version of the characterization of proper (locally dualizable) spaces via the canonicity of their global dualizing complexes.
As a consequence, we show (Corollary \ref{c-proper=proper})  that every c-proper map between proper and dualizable spaces is proper. We also demonstrate the abundance of proper subspaces: for every locally dualizable space $X$ and every $p\in X$, the (open) subspace $U_p-\{p\}$ is proper (Corollary \ref{abundance}). Note that if $X$ is an abstract simplicial complex, then $U_p-\{p\}$ is  isomorphic to  the link of $p$. A final, unexpected result of this section—for which we do not know of a classical geometric analog—is given by Theorem \ref{properversusCM}: a locally dualizable space $X$ is proper if and only if its dual space $X^{\rm op}$ is Cohen--Macaulay with an invertible canonical sheaf.

Finally, section \ref{section-examples} explores two basic examples. On the one hand we prove that proper maps between abstract simplicial complexes are precisely simplicial maps (Theorem \ref{simplicial=proper} and Remark \ref{simplicial=proper2}). On the other hand we show that a connected $1$-dimensional space is proper if and only if each of its irreducible components   is a  projective line; consequently, a $1$-dimensional proper space is equivalent to an undirected multigraph without loops (Proposition \ref{completecurves} and Remark \ref{completecurves2}).

This paper stands as a delayed and grateful response to a question raised by the anonymous referee of \cite{ST2}. In their report, the referee inquired about an $f^!$ theory ($f$ being a continuous map between finite spaces) with properties and behavior analogous to the $f^!$ theory for schemes or locally compact Hausdorff spaces. As they pointed out: ``In the context of a map $f \colon X \to Y$ of finite spaces, I am unsure what the correct analog of $f^!$ should be since the analogs of separated or proper maps with certain desired properties is hard to guess or may not even exist''. This paper demonstrates that, at the very least, the analog of proper maps \emph{does} exist, thus providing the first step toward developing such an $f^!$ theory on finite spaces.

\section{Preliminaries and notations}\label{prelim-section}

\noindent{\S\bf \  Finite spaces.} Throughout this paper, a finite space means a finite and $T_0$ topological space. All topological spaces considered in this paper will be assumed, by default, to be finite spaces. As  is well known, a finite space is equivalent to a finite poset, but we shall preferably adopt the topological point of view. For each $p\in X$ we  denote
\[ \aligned U_p &=\text{ smallest open subset of } X \text{ containing } p, 
\\ C_p & = \text{ smallest closed subset of } X \text{ containing } p =\text{ closure of } \{p\}.\endaligned\]
We also introduce the notation $$U_p^*:=U_p\negmedspace - \negmedspace \{p\}\quad, \quad \partial C_p :=C_p\negmedspace - \negmedspace \{p\}.$$

The partial order on $X$ is defined by
\[p\leq q \Leftrightarrow C_p\subseteq C_q \Leftrightarrow U_p\supseteq U_q.\] A subset $U$ of $X$ is open if and only if it is an upper set (increasing): $p\in U$ and $q\geq p$ implies $q\in U$. Analogously, a subset $C$ of $X$ is closed if and only if it is a lower set (decreasing): $p\in C$ and $q\leq p$ implies $q\in C$. Thus, one has:
\[ \aligned U_p &=\left\{ x\in X: x\geq p\right\}  
\\ C_p &=\left\{ x\in X: x\leq p\right\}.\endaligned\] A point $p\in X$ is closed if and only if it is minimal; analogously, $p$ is open if and only if it is maximal. Open points will be referred to as {\em generic points} of $X$, since their closures constitute the irreducible components of $X$.

A map $f\colon X\to Y$ between finite spaces is continuous if and only if it is monotone (order-preserving): $p\leq q\implies f(p)\leq f(q)$. 

The dual space of $X$ is denoted by $X^{\text{\rm op}}$. It is the same underlying set equipped with the dual topology (i.e., the reversed partial order), meaning that  an open subset of $X^{\text{\rm op}}$ is a closed subset of $X$. For each $p\in X$, we  denote  $\wh p\in X^{\text{\rm op}}$ the same element realized in the dual space.

\begin{defn} {\rm We say that a finite space $X$ is {\em local} if it has a unique closed point (i.e., $X$ has a minimum) which we shall usually denote by $\0$. The open complement $X\negmedspace - \negmedspace\{\0\}$ will be usually denoted by $X^*$. For every point $p\in X$, $U_p$ is local  with closed point $p$, and $C_p$ is irreducible with generic point $p$.}
\end{defn}

\begin{defn}{\rm  The {\em dimension}  of a finite space $X$, denoted by $\dim X$, is the maximal length of the chains  
\[ C_0 \subset C_1\subset \cdots\subset C_n\] of irreducible closed subsets of $X$; equivalently, it is the maximal length of the chains of points $p_0 <p_1 <\cdots < p_n$ in $X$. We  also define $\dim \emptyset=-1$. For every $p\in X$, we denote $\codim(p,X)=\dim U_p$, which is the maximal length of the chains of points of $X$ starting at $p$.} 
\end{defn}

\begin{defn}\label{dimension+pure} {\rm A finite space  $X$ is called {\em pure} if all  maximal chains $p_0 <p_1 <\cdots < p_n$ have length equal to $\dim X$. }\end{defn}

\begin{defn}\label{catenary}{\rm For every $x <y$, we  denote
\[\aligned  (x,y)&:=U_x^*\cap  \partial C_y=\{ p\in X: x<p<y\} 
\\   [x,y]&:= U_x  \cap  C_y =\{ p\in X: x\leq p\leq y\}.\endaligned \] We say that $X$ is {\em catenary} if every interval $[x,y]$ is pure; that is, every maximal chain from $x$ to $y$ has length equal to $\dim [x,y]$. 
}
\end{defn}

\noindent{\S\bf \  Affine and projective spaces;  spheres.}
\medskip

{\em The affine space}. We  denote $$\A^1_{\FF_1}=\{\0,\un\} , \text{ with }  \0 <\un.$$  and set $$\A^n_{\FF_1}= \A^1_{\FF_1}\times\overset n\cdots\times \A^1_{\FF_1}.$$ This is a local and irreducible space of dimension $n$ with closed point $\0=(\0,\dots,\0)$ and generic point $\un=(\un,\dots,\un)$. It can be identified with the poset of subsets of a set with $n$ elements, ordered by inclusion. For every $p\in\A^n_{\FF_1}$, one has
\[ C_p=\A^{\dim C_p}_{\FF_1}\quad ,\quad U_p=\A^{\dim U_p}_{\FF_1}.\]
\medskip

The {\em $n$-dimensional projective space} $\PP^n_{\FF_1}$ is defined as:
\[ \PP^n_{\FF_1}=\A^{n+1}_{\FF_1}-\{\0\}.
\] It is an irreducible open subset of $\A^{n+1}_{\FF_1}$ of dimension $n$ with generic point $\un$.  It can be identified with the poset of non-empty subsets of a set with $n+1$ elements. For every $p\in\PP^n_{\FF_1}$, one has
\[ C_p=\PP^{\dim C_p}_{\FF_1}\quad ,\quad U_p=\A^{\dim U_p}_{\FF_1}.\]

 A {\em simplicial complex} $K$ (with at most $n$ vertices) is a closed subset of $\A^n_{\FF_1}$. A {\em projective simplicial complex} is a closed subset of $\PP^{n}_{\FF_1}$. Evidently, if $K$ is a simplicial complex, then $K^*=K\negmedspace - \negmedspace\{\0\}$ is a projective simplicial complex; conversely, adding $\0$ to a projective simplicial complex yields a simplicial complex. Projective simplicial complexes are standardly called \emph{abstract simplicial complexes} in the combinatorics literature.  A finite space $X$ is {\em locally simplicial} if $U_p$ is a simplicial complex for every $p\in X$. Every simplicial complex and every projective simplicial complex is locally simplicial. A finite space $X$ is locally simplicial if and only if $X^{\text{\rm op}}$ is locally simplicial (this follows from the fact that $(\A^n_{\FF_1})^{\text{\rm op}}\simeq \A^n_{\FF_1}$). Any 1-dimensional finite space is locally simplicial.  
 
 \begin{ex} {\rm For every finite space $X$, its {\em barycentric subdivision} $\beta X$ is the set of non-empty, totally ordered subsets of $X$. It is a closed subset of the poset of all non-empty subsets of $X$. Hence, $\beta X$ is a projective simplicial complex.}
\end{ex}
 
 There exist noteworthy finite spaces that are not locally simplicial. For example:

\begin{defn}\label{spheres}{\rm The {\em $n$-dimensional sphere} $\Sa^n$ is the finite space with $2n+2$ elements $$\Sa^n=\{p_0,\dots,p_n,q_0,\dots ,q_n\},$$ endowed with the partial order:
\[ p_i<p_j, q_j \text{ and } q_i < q_j,p_j \, ,\, \text{ for all } i<j.\]}\end{defn} It has dimension $n$ and constitutes  the minimal finite model of the standard
$n$-dimensional sphere $S^n$ (see \cite{Barmak}). The generic points of $\Sa^n$ are $p_n,q_n$. Thus, $\Sa^n$ has two irreducible components $C_{p_n},C_{q_n}$, whose intersection is an $(n-1)$-dimensional sphere. For every $p\in\Sa^n$, one has:
\[ U_p^*=\Sa^{\dim U_p^*}   \quad  , \quad \partial C_p=\Sa^{\dim \partial C_p}\qquad (\text{where we define }  \Sa^{-1}:=\emptyset).\]

The sphere $\Sa^n$ is not locally simplicial for $n>1$. While $\Sa^1$ is locally simplicial, it is not a simplicial complex.
\medskip

%\begin{defn} {\rm A morphism $f\colon X\to Y$ is {\em locally simplicial} if for any $x\in X$, the map $f_{\vert U_x}\colon U_x\to U_{f(x)}$ factors through a closed immersion $U_x\hookrightarrow U_{f(x)}\times \A^n_{\FF_1}$ and the natural projection $U_{f(x)}\times \A^n_{\FF_1} \to U_{f(x)}$. 
%}\end{defn}

Borrowing terminology from algebraic geometry, we say that $f\colon X\to Y$ is a {\em proyective morphism} if, locally on $Y$, it factors through a closed immersion $X\hookrightarrow Y\times\PP^n_{\FF_1}$ followed by the natural projection $Y\times\PP^n_{\FF_1}\to Y$. This definition features a phenomenon that does not occur in algebraic geometry: the composition of projective morphisms is not necessarily projective. This stems from the fact that the product $\PP^n_{\FF_1}\times \PP^m_{\FF_1}$ of two projective spaces is not a projective simplicial complex (i.e., it does not admit  a closed immersion into a projective space). As we shall see, proper morphisms will resolve this issue.
%It is clear that every projective morphism is locally simplical. 
A more general concept is given by  the following:

\begin{defn}{\rm  
A morphism $f\colon X\to Y$ is {\em locally of finite type} if, for every $x\in X$, the restriction $f_{\vert U_x}\colon U_x\to U_{f(x)}$ factors through a closed immersion $U_x\hookrightarrow U_{f(x)}\times T$ followed by  the natural projection $U_{f(x)}\times T \to U_{f(x)}$, for some finite space $T$.}
\end{defn} 

Every closed or open immersion is locally of finite type, as is the natural projection $X\times Y\to Y$. The composition of morphisms locally of finite type is also locally of finite type. Moreover, every projective morphism is locally of finite type.

\bigskip
\noindent{\S\bf \  Local and global dualizing complexes.} We recall here the definitions and main properties of global and local dualizing complexes. We first introduce some standard notations.

We  denote by $\mathrm {D}(\ZZ)$ the derived category of the category $C(\ZZ)$ of complexes of abelian groups. 

For every $E\in \mathrm{D}(\ZZ)$, we denote its derived dual by $E^\vee:=\RR\Hom^\pun_\ZZ(E,\ZZ)$.

We shall denote by $C(X)$ the category of complexes of sheaves of abelian groups on $X$ and by $\mathrm{D}(X)$ its derived category.  For every complex $F\in C(X)$ and every locally closed subset $S$ of $X$, we  denote by $F_S$ the complex $F$ supported in $S$; it is the unique complex satisfying
\[ (F_S)_{\vert S}=F_{\vert S}\quad , \quad (F_S)_{\vert X-S}=0.\] 
To avoid confusion, for each $p\in X$, $F_p$ will denote  the stalk of $F$ at $p$, whereas $F_{\{p\}}$ will denote the complex $F$ supported in $\{p\}$.

For any two complexes $F,G\in \mathrm{D}(X)$, we write $\RR\Hom_X^\pun(F,G)\in \mathrm{D}(\ZZ)$ for the derived complex of homomorphisms,   $\RR\HHom_X^\pun(F,G)\in \mathrm{D}(X)$ for the derived complex of sheaves of homomorphisms, and $F\overset\LL\otimes G$ for the derived tensor product. 

We denote by    $\RR\Gamma(X,-)$  the right derived functor  of the functor of sections $\Gamma(X,-)$. For every complex $F\in\mathrm{D}(X)$, the groups
\[ H^i(X,F):=H^i[\RR\Gamma(X,F)] \] are the cohomology  groups of $F$. For a sheaf $F$ on $X$, $H^i(X,F)$ vanishes for all $i\notin[0,\dim X]$. It is a trivial but fundamental fact regarding finite spaces that $\Gamma(U_x,F)=F_x$, and hence $\Gamma(U_x,-)$ is an exact functor.

A classical result that we shall utilize later  is that every finite space has the same integral cohomology as  its dual space  (see, for example, \cite[Cor. 4.2.6]{Sanchoetal} for a proof)).
\begin{prop}\label{X=Xop} For every finite space $X$, one has $$\RR\Gamma(X,\ZZ)=\RR\Gamma(X^{\text{\rm op}},\ZZ).$$
\end{prop}

For every closed subset $Y$ of $X$, we  denote $\RR\Gamma_Y(X,-)$ the right derived functor of the functor of sections with support in $Y$ and $\RR\underline\Gamma_Y $ its sheafified version. There are natural isomorphisms $\RR\Hom_X^\pun(\ZZ_Y,F)=\RR\Gamma_Y(X,F)$ and $\RR\HHom_X^\pun(\ZZ_Y,F)=\RR\underline\Gamma_Y F $. For every open subset $U\subseteq X$ containing $Y$, one has $\RR\Gamma_Y(X,F)=\RR\Gamma_Y(U,F)$, which will be referred to as the excision property. 

For every continuous map $f\colon X\to Y$, $\RR f_*\colon \mathrm{D}(X)\to \mathrm{D}(Y)$ denotes the right derived functor  of the direct image $f_*$. It is left adjoint of the inverse image functor $f^{-1}$.

\begin{defn}\label{defglobaldualizing} {\rm Let $f\colon X\to Y$ be a continuous map. We  denote by
\[ f^\times\colon \mathrm{D}(Y)\to \mathrm{D}(X)\]
the right adjoint of the functor $\RR f_*\colon \mathrm{D}(X)\to \mathrm{D}(Y)$ (see \cite{Na}, \cite{Curry}, \cite{Sanchoetal}).  
 Thus, for every $F\in \mathrm{D}(X)$ and every $G\in \mathrm{D}(Y)$, there is a natural isomorphism
\[\RR\Hom_X^\pun(F,f^\times G )=\RR\Hom_Y^\pun(\RR f_*F,G).\] If $Y$ is   a single point, we set $D_X=f^\times\ZZ$ and refer to it as the  {\em global dualizing complex} of $X$. By definition,
\[\RR\Hom_X^\pun(F,D_X)= \RR\Gamma(X,F)^\vee.\]}\end{defn}

\begin{defn}{\rm  Let $f\colon X\to Y$ be a continuous map and let $Z$ be a closed subset of $X$. For every sheaf $F$ on $X$, we define the sheaf $f_*^ZF$ on $Y$ via
\[ (f_*^ZF)(V)=\Gamma_{Z\cap f^{-1}(V)}(f^{-1}(V),F)\] for every open subset $V$ of $Y$. In other words,
\[ f_*^ZF=f_*\HHom_X(\ZZ_Z,F)=f_*{\underline\Gamma}_ZF.\] Its right derived functor
\[ \RR f_*^Z\colon \mathrm{D}(X)\to \mathrm{D}(Y) \] can be identified as $\RR f_*^ZF=\RR f_* \RR\HHom_X^\pun(\ZZ_Z,F)=\RR f_*\RR {\underline\Gamma}_ZF$. This functor also has a right adjoint (\cite{ST2}),  denoted by
\[ f^\times_Z\colon \mathrm{D}(Y)\to \mathrm{D}(X).\] If $Y$ is a single point, we set $D_X^Z:=f^\times_Z\ZZ$ and call it the \emph{local dualizing complex of $X$ along $Z$}.}
\end{defn}

The following particular case will be of special relevance:
\begin{defn}\label{deflocaldualizing}{\rm Let $X$ be a finite local space, and let $\mathbf{0} \in X$ be its unique closed point. The complex $D_X^\0$ will be called {\em the local dualizing complex of $X$}. By definition,
\[\RR\Hom_X^\pun(F,D_X^\0)= \RR\Gamma_\0(X,F)^\vee=\RR\Hom_X^\pun(\ZZ_{\{\0\}},F)^\vee.\]
}
\end{defn} 

\begin{prop}[Transitivity of $f^\times_Z$]\label{transitivity} Let $  X\overset f\to Y\overset g\to T$ be continuous maps, let $Z$ be a closed subset of $X$, and let $W$ be a closed subset of $Y$. There is a natural isomorphism
\[ f_Z^\times\circ g_W^\times = (g\circ f)_{Z\cap f^{-1}(W)}^\times.\]
\end{prop}

\begin{proof} By adjunction, it suffices to show that $\RR g_*^W\circ\RR f_*^Z =\RR (g\circ f)_*^{Z\cap f^{-1}(W)}$. Indeed, for every $F\in \mathrm{D}(X)$, we have:
\[\begin{aligned} 
(\mathbb{R} g_*^W \circ \mathbb{R} f_*^Z)(F) 
&= \mathbb{R} g_* \mathbb{R}\mathcal{H}om_Y^\bullet(\mathbb{Z}_W, \mathbb{R} f_* \mathbb{R}\mathcal{H}om_X^\bullet(\mathbb{Z}_Z, F)) \\ 
&= \mathbb{R} g_* \mathbb{R} f_* \mathbb{R}\mathcal{H}om_X^\bullet(\mathbb{Z}_{f^{-1}(W)}, \mathbb{R}\mathcal{H}om_X^\bullet(\mathbb{Z}_Z, F)) \\ 
&= \mathbb{R} g_* \mathbb{R} f_* \mathbb{R}\mathcal{H}om_X^\bullet(\mathbb{Z}_{f^{-1}(W)} \overset{\mathbb{L}}{\otimes} \mathbb{Z}_Z, F) \\ 
&= \mathbb{R}(g \circ f)_* \mathbb{R}\mathcal{H}om_X^\bullet(\mathbb{Z}_{f^{-1}(W) \cap Z}, F) \\ 
&= \mathbb{R}(g \circ f)_*^{Z \cap f^{-1}(W)}(F),
\end{aligned}
\]
where we have applied the sheafified adjunction $f^{-1} \dashv \mathbb{R} f_*$ in the second isomorphism, and the sheafified tensor-hom adjunction $\overset{\mathbb{L}}{\otimes} \dashv \mathbb{R}\mathcal{H}om$ in the third.

%\[\aligned (\RR g_*^W\circ\RR f_*^Z )(F) &= \RR g_*\RR\HHom_Y^\pun(\ZZ_W,\RR f_*\RR\HHom_X^\pun(\ZZ_Z,F)) 
%\\ (\text{sheafified adjunction }f^{-1}\leftrightarrow \RR f_* ) &= \RR g_*\RR f_*\RR\HHom_X^\pun(\ZZ_{f^{-1}(W)},\RR\HHom_X^\pun(\ZZ_Z,F)) 
%\\ (\text{sheafified adjunction }\overset\LL\otimes\leftrightarrow\RR\HHom ) &= \RR g_*\RR f_*\RR\HHom_X^\pun(\ZZ_{f^{-1}(W)}\overset\LL\otimes\ZZ_Z,F) 
%\\ &= \RR (g\circ f)_*\RR\HHom_X^\pun(\ZZ_{f^{-1}(W)\cap Z} ,F) 
%\\ &= \RR (g\circ f)_*^{Z\cap f^{-1}(W)}(F).
%\endaligned.\]
\end{proof}

\medskip
\noindent{\S\bf \ Dualizable and locally dualizable spaces.}
\medskip

For every finite space $X$, we denote by $\mathrm{D_c}(X)$   the full subcategory of $\mathrm{D}(X)$ consisting of complexes $F$   whose cohomology sheaves  $H^i(F)$ are finitely generated (i.e., their stalks are finitely generated abelian groups), and by $\mathrm{D^b_c}(X)$ the full subcategory of complexes in $\mathrm{D_c}(X)$ with bounded cohomology.

\begin{defn}\label{defcanonical}{\rm  A {\em canonical complex} on $X$ is a complex $\Omega\in \mathrm{D^b_c}(X)$ such that, for every $F\in \mathrm{D_c}(X)$, the natural morphism  
 \[F\to \RR\HHom_X^\pun(\RR\HHom_X^\pun(F,\Omega),\Omega)\] is an isomorphism. 
 A finite space is called {\em dualizable} if a canonical complex on $X$ exists. A finite space is called {\em locally dualizable} if $U_x$ is dualizable for every $x\in X$.
}\end{defn}

On a connected finite space, a canonical complex, if it exists,  is unique up to a shift and tensoring by an invertible sheaf (\cite[Thm. 4.8]{ST2}). Every open subset $U\hookrightarrow X$ of a dualizable (resp. locally dualizable) space is dualizable (resp. locally dualizable); in fact,  if $\Omega$ is a canonical complex on $X$, then $\Omega_{\vert U}$ is a canonical complex on $U$ (\cite[Thm. 4.5, (1)]{ST2}). Every closed subset $ j\colon C\hookrightarrow X$ of a dualizable (resp. locally dualizable) space is dualizable (resp. locally dualizable); in fact,  if $\Omega$ is a canonical complex on $X$, then $j^{-1}\RR\HHom_X^\pun(\ZZ_C,\Omega)$ is a canonical complex on $C$ (\cite[Thm. 4.5, (2)]{ST2}). The product of two dualizable (resp. locally dualizable) spaces is dualizable (resp. locally dualizable) (\cite[Prop. 4.24]{ST2}). 

\begin{defn}\label{coh-sphere} {\rm The {\em reduced cohomology} $\RR\Gamma_\text{\rm red}(X,\ZZ) $ is defined as the cone of the natural morphism $\ZZ\to\RR\Gamma(X,\ZZ)$, and $H^i_{\text{\rm red}} (X,\ZZ):= H^i[\RR\Gamma_\text{\rm red}(X,\ZZ)]$ denote the reduced cohomology groups.  Note that  if $X$ is empty, then $H^{-1}_{\text{\rm red}}(X,\ZZ)=\ZZ$ and $H^i_{\text{\rm red}}(X,\ZZ)=0$ for all $i\neq -1$.

We say that $X$ is a {\em cohomological sphere} if $\RR\Gamma_\text{\rm red}(X,\ZZ)=\ZZ[-\dim X]$; that is, if $H^i_{\text{\rm red}} (X,\ZZ)=0$ for all $i\neq\dim X$ and $H^{\dim X}_{\text{\rm red}} (X,\ZZ)=\ZZ$.}
\end{defn}

\begin{ejems} {\rm $\,$\medskip

  (1) The affine and projective spaces  $\AAA^n_{\FF_1}$ and $ \PP^n_{\FF_1}$ are dualizable, with canonical complex given by $\ZZ_{\{\un\}}$ (see \cite{ST}). Consequently, every simplicial or projective simplical complex is dualizable, and every locally simplical space is locally dualizable. 

(2) The spheres $\mathbb S^n$ are dualizable, and $\ZZ$ is a canonical complex. Indeed, by \cite[Proposition 4.4]{ST2}, it suffices to show   that the sheaves $\ZZ_{\{p\}}$ are $\ZZ$-reflexive for every $p\in\Sa^n$; that is, the natural morphism
\[ \ZZ_{\{p\}}\to\RR\HHom^\pun_{\Sa^n} (\RR\HHom^\pun_{\Sa^n} (\ZZ_{\{p\}},\ZZ),\ZZ)\] is an isomorphism. Following the notation of Definition \ref{spheres}, let us assume without loss of generality  that $p=p_i$ (the case $p=q_i$ is analogous) and  denote $q=q_i$. One has (\cite[Proposition 2.5, (1)]{ST2}) $$\RR\HHom^\pun_{\Sa^n} (\ZZ_{\{p\}},\ZZ)=\RR\Gamma_p(U_p,\ZZ)\otimes_\ZZ \ZZ_{C_p} = \ZZ_{C_p}[-\dim U_p]$$ 
where the last equality follows from  $\RR\Gamma_p(U_p,\ZZ)=\ZZ[-\dim U_p]$, since $U_p^*=\Sa^{\dim U_p^*}$. To conclude, it suffices to show that $\RR\HHom^\pun_{\Sa^n} (\ZZ_{C_p},\ZZ)=\ZZ_{\{p\}}[-\dim U_p]$. The complex $\RR\HHom^\pun_{\Sa^n} (\ZZ_{C_p},\ZZ)$ is supported in $C_p$. For every $x<p$, we have
\[\RR\HHom^\pun_{\Sa^n} (\ZZ_{C_p},\ZZ)_x= \RR\Hom^\pun_{U_x} (\ZZ_{U_x\cap C_p},\ZZ) =\RR\Gamma_{U_x\cap C_p}(U_x,\ZZ)=0\] where the final equality vanishes because  $U_x$ and $U_x-(U_x\cap C_p)$ are acyclic and connected (note that $U_x-(U_x\cap C_p)=\Sa^n-C_p=U_q$). Hence, $\RR\HHom^\pun_{\Sa^n} (\ZZ_{C_p},\ZZ)$ is supported strictly at $\{p\}$, and its stalk at $p$ is $\RR\Gamma_p(U_p,\ZZ)=\ZZ[-\dim U_p]$. Thus, $\RR\HHom^\pun_{\Sa^n} (\ZZ_{C_p},\ZZ)=\ZZ_{\{p\}}[-\dim U_p]$.}

\end{ejems}

Locally dualizable spaces are characterized by the following result:
\begin{thm}\cite[Cor. 4.17]{ST2}\label{locallydualizable} A finite space  $X$ is locally dualizable if and only if:
\begin{enumerate}\item $X$ is catenary (Definition \ref{catenary}): every closed interval $[x,x']$ is pure.
\item For every $x<x'$, the open interval $(x,x')$ is a cohomological sphere (Definition \ref{coh-sphere}).
\end{enumerate}
\end{thm}

If $X$ is   local, the  unique candidate (up to a shift) for a  canonical complex is the local dualizing complex:

\begin{thm}\cite[Thm. 4.16]{ST2}\label{dualizablelocal} Let $X$ be a local space with unique closed point $\0$. If $\Omega$ is a canonical complex, then $\Omega=D_X^\0[r]$ for a unique $r\in\ZZ$. Thus, $X$ is dualizable if and only if $D_X^\0$ is a canonical complex.
\end{thm}

\begin{rem}\label{remark}{\rm Let $X$ be a local and dualizable space with closed point $\0$, and let $\Omega$ be a canonical complex. Thus, we have an isomorphism $\psi\colon\Omega\overset\sim\to D_{X}^\0[r]$ for some integer $r$. Then,  
\[ \Hom_{\mathrm{D}(X)} (\Omega ,D_{X}^\0)= \left\{\aligned 0,\, &\text{ if }r\neq 0,\\ \ZZ\cdot\psi, &\text{ if } r=0.\endaligned\right.\] Hence, a morphism $  \Omega \overset\phi\to D_{X}^{\0}$ is an isomorphism if and only if $\RR\Gamma_\0(\phi)\colon \RR\Gamma_\0(X,\Omega) \to \RR\Gamma_\0(X,D_{X}^{\0})$ is an isomorphism.}
\end{rem}

\begin{proof} Since 
\[ \RR\Hom^\pun_X(D_{X}^{\0}[r],D_{X}^{\0})=\RR\Hom_X^\pun(\ZZ_{\{\0\}}, \ZZ_{\{\0\}})^{\vee\vee} [-r]=\ZZ[-r]\] we obtain that $\Hom_{\mathrm{D}(X)}(\Omega ,D_{X}^{\0})= \left\{\aligned 0,\, &\text{ if }r\neq 0,\\ \ZZ\cdot\psi, &\text{ if } r=0\endaligned\right.$. The conclusion follows directly  from the identity $\RR\Gamma_\0(X,D_{X}^{\0})=\ZZ$.
\end{proof}

If $X$ is   irreducible, the  unique candidate (up to a shift) for a  canonical complex is $\ZZ_{\{g\}}$, where $g$ is the generic point:

\begin{thm}\cite[Thm. 4.18]{ST2} \label{dualizableirreducible} Let $X$ be an irreducible space with generic point $g$. Then $X$ is dualizable if and only if it is locally dualizable. If $\Omega$ is a canonical complex on $X$, then $\Omega=\ZZ_{\{g\}}[r]$ for a unique $r\in\ZZ$. Thus, $X$ is dualizable if and only if $\ZZ_{\{g\}}$ is a canonical complex.
\end{thm}

A consequence of Theorem \ref{locallydualizable} is the invariance of local dualizability under taking the dual space $(-)^{\text{\rm op}}$:

\begin{cor}\label{loc-dual-op} The following conditions are equivalent:
\begin{enumerate} \item $X$ is locally dualizable.
\item $X^{\rm op}$ is locally dualizable.
\item For every $x\in X$, $C_x$ is dualizable.
\item Every irreducible component of $X$ is dualizable.
\end{enumerate}

\begin{proof} The equivalence of (1) and (2) follows from Theorem \ref{locallydualizable}, taking into account that
\[ [x,y]^{\text{\rm op}}=[\wh y,\wh x]\quad , \quad (x,y)^{\text{\rm op}}=(\wh y,\wh x)\] and that $\RR\Gamma(X,\ZZ)=\RR\Gamma(X^{\text{\rm op}},\ZZ)$ for every finite space $X$ (Proposition \ref{X=Xop}).

For (2)$\Rightarrow $(3), since $(C_x)^{\text{\rm op}}=U_{\wh x}$, we obtain that $C_x$ is locally dualizable, and hence dualizable by Theorem \ref{dualizableirreducible}.

The implication (3)$\Rightarrow $(4)  is immediate.  Finally, for (4)$\Rightarrow $(2), let $X=X_1\cup\dots\cup X_r$ be the decomposition into irreducible components. Since each $X_i$ is dualizable, $(X_i)^{\text{\rm op}}$ is locally dualizable. It follows that  $X^{\text{\rm op}}$ is locally dualizable, as it is covered by the open subsets $(X_i)^{\text{\rm op}}$.
\end{proof}

\end{cor}

\begin{rem}{\rm In the combinatorics literature, a {\em simplicial poset} is a poset with a minimal element such that every closed interval is a Boolean algebra. That is, it is a local space $X$ such that $C_p\simeq \A^{d_p}_{\FF_1}$ for every $p$. By Corollary \ref{loc-dual-op}, every such space is locally dualizable. More generally, every space $X$ whose irreducible components are locally simplicial is locally dualizable.}
\end{rem}

\medskip
\noindent{\S\bf \ Cohen--Macaulayness.}
\medskip

\begin{defn} {\rm Let $X$ be a dualizable local space with unique closed point   $\0$. We say that $X$ is {\em Cohen--Macaulay} if $D_X^{\0}=\omega_X[\dim X]$ for some sheaf $\omega_X$ on $X$, which is called {\em canonical sheaf} of $X$.

In general, a finite space $X$ is Cohen--Macaulay if it is locally dualizable and $U_p$ is Cohen--Macaulay for every $p\in X$. If $X$ is connected, dualizable, and  Cohen--Macaulay, and if $\Omega$ is a canonical complex on $X$, then $\Omega=\omega_X[r]$, where $\omega_X$ is  a canonical sheaf on $X$. This canonical sheaf  is unique up to tensoring by an invertible sheaf, and it satisfies  $(\omega_X)_{\vert U_p}=\omega_{U_p}$ for every $p\in X$.
}\end{defn}

In the simplicial framework, Cohen--Macaulayness is closely related to Cohen--Macaulayness in the sense of Baclawski \cite{B} or Stanley--Reisner \cite{BH}. More precisely (see \cite{ST2}): let $K$ be a simplicial complex and $K^*=K-\{\0\}$ its associated projective simplicial complex. Then $K$ is Cohen--Macaulay if and only if $K^*$ is Cohen--Macaulay in the sense of Baclawsky (or equivalently,  in the Stanley--Reisner sense). Furhtermore, $K^*$ is Cohen--Macaulay if and only if $K^*$ is ACM (almost Cohen--Macaulay) in the sense of Baclawski \cite{B}.

The cohomological characterization of   Cohen-Macaulayness is given by the following result:  

\begin{thm}\cite[Thm. 5.18 and Cor. 5.20]{ST2}\label{CM} A locally dualizable space $X$ is Cohen--Macaulay if and only if, for every $p\in X$, one has:
\[\aligned  H^i_p(U_p,\ZZ) &=0 \text{ for }i\neq\dim U_p,\quad  \text{and}\\   H^{\dim U_p}_p(U_p,\ZZ)&\quad  \text{is torsion free (i.e., a finitely generated free abelian group).}\endaligned\] Moreover, in this case, 
\[(\omega_{U_p})_q= H^{\dim U_q}_q(U_q,\ZZ)^* \] for every $p\in X$ and every $q\in U_p$, where $(-)^*=\Hom_\ZZ(-,\ZZ)$.
\end{thm}

\begin{prop}\label{partial X} Let $X$ be an irreducible and dualizable finite space. Then $\partial X$ is dualizable, and $\ZZ$ is a canonical complex on it. In particular, $\partial X$ is Cohen-Macaulay, and $\ZZ$ is a canonical sheaf. 
\end{prop}

\begin{proof}  Let $g$ be the generic point of $X$. Since  $X$ is dualizable,   $\ZZ_{\{g\}}$ is a canonical complex on $X$. Therefore, $\partial X$ is dualizable with canonical complex given by $i^{-1}\RR\HHom^\pun_X(\ZZ_{\partial X},\ZZ_{\{g\}})$, where $i\colon \partial X\hookrightarrow X$ denotes the closed immersion.  %(see \cite{?}). 

Applying the functor  $\RR\HHom^\pun_X(-,\ZZ_{\{g\}})$ to the short exact sequence
\[ 0\to \ZZ_{\{g\}}\to\ZZ\to \ZZ_{\partial X}\to 0,\] we obtain the exact triangle
\[ \RR\HHom^\pun_X(\ZZ_{\partial X},\ZZ_{\{g\}})\to \ZZ_{\{g\}} \to \ZZ.\] This implies that $\RR\HHom^\pun_X(\ZZ_{\partial X},\ZZ_{\{g\}})=\ZZ_{\partial X}[-1]$, and consequently $i^{-1}\RR\HHom^\pun_X(\ZZ_{\partial X},\ZZ_{\{g\}})=\ZZ[-1]$. Thus,  $\ZZ$ is a canonical complex on $\partial X$. This establishes that $\partial X$ is Cohen-Macaulay and that $\ZZ$ is a canonical sheaf.
\end{proof}

\begin{cor}\label{Gor} Let $X$ be a locally dualizable space. The following conditions are equivalent:
\begin{enumerate} 
\item For every $p\in X$, $U_p^*$ is a cohomological sphere.
\item $X$ is Cohen--Macaulay, and $(\omega_{U_p})_p\simeq \ZZ$ for every $p\in X$.
\item $X$ is Cohen--Macaulay, and $\omega_{U_p}\simeq \ZZ$ for every $p\in X$.
\end{enumerate}
\end{cor}
\begin{proof} Since all conditions are local, we may assume without loss of generality that $X$ is a local and dualizable space.  For each $p\in X$, one has
\[ H^i_p(U_p,\ZZ)=H^{i-1}_{\text{\rm red}}(U_p^*,\ZZ),\text{ for all }i.\]
This follows directly from the exact triangle of local cohomology; see \cite[Lemma 1.6]{ST2} for explicit  details. 

The equivalence of (1) and (2) follows immediately  from Theorem \ref{CM}. The implication (3) implies (2) is trivial. Finally, let us show that (2) implies (3). Assume that $X$ is local, Cohen--Macaulay, and $(\omega_X)_p\simeq \ZZ$ for every $p\in X$; we aim to conclude that $ \omega_X \simeq \ZZ$. 

If $\dim X=0$, the result immediate. If $\dim X=1$, let $\0$ be the closed point. Since $X^*$ is a $0$-dimensional cohomological sphere, it consists of exactly two points,  $X^*=\{ q_1,g_2\}$ (i.e., $X$ has two generic points). This implies that $X\simeq \partial\A^2_{\FF_1}$ and the claim follows from Proposition \ref{partial X}. 

Assume now that  $\dim X>1$,  and let $j\colon X^*\hookrightarrow X$ be the open immersion. Then $\omega_{X^*}:=(\omega_X)_{\vert X^*}$ is a canonical sheaf on $X^*$. The natural morphism $\omega_X\to j_*\omega_{X^*}$ is an isomorphism because $H^0_{\0}(\omega_X)=H^1_{\0}(\omega_X)=0$ (see \cite[Thm. 5.18]{ST2}). By the induction hypothesis, $\omega_{X^*}$ is an invertible sheaf. To show that  $\omega_{X^*}\simeq \ZZ$, it suffices to verify that $H^0(X^*,\omega_{X^*})\neq 0$. Since $\omega_{X^*}$ is a dualizing sheaf (see \cite[Thm. 5.24, (B)]{ST2}), we have
\[H^0(X^*,\omega_{X^*})=H^{\dim X^*}(X^*,\ZZ)^*=\ZZ\] where the last isomorphism holds because $X^*$ is a cohomological sphere. 

Furthermore, the natural morphism $\ZZ\to j_*\ZZ$ is an isomorphism because $H^0_{\0}(\ZZ)=H^1_{\0}(\ZZ)=0$. Combining these facts, we obtain  $\omega_X\simeq\ZZ$.

\end{proof}

\bigskip
\noindent{\S\bf \ Cohomologically proper maps.} Cohomologically proper maps were introduced in \cite{Sanchoetal} in order to characterize those maps satisfying the projection formula, base change, local form of duality, etc. Very briefly, a cohomologically proper map $f\colon X\to Y$ between finite spaces is a map such that $\RR f_*$ enjoys the same cohomological properties as the direct image functor  when $f$ is a proper map between locally compact and Haussdorff spaces. We summarize here  the results concerning these maps that  will be utilized throughout this work  (see \cite{Sanchoetal} for their proofs). The only novel result provided here  is Proposition \ref{cohpropsup}.

\begin{defn}\label{c-trivial}{\rm We say that a finite space $X$ is {\em cohomologically trivial} if $\RR\Gamma_{\text{\rm red}}(X,\ZZ)=0$. In other words, $H^i(X,\ZZ)=0$ for all $i\neq 0$ and $H^0(X,\ZZ)=\ZZ$ (which implies that $X$ is non-empty and connected).
}
\end{defn}

\begin{defn}\label{c-proper}{\rm A continuous map $f\colon X\to Y$ is called {\em cohomologically proper} ({\em c-proper}, for short) if, for every $p\in X$ the map $f_{\vert C_p}\colon C_p\to C_{f(p)}$ has cohomologically trivial fibres.}
\end{defn}

\begin{rem}{\rm In \cite{Sanchoetal}, a cohomologically proper map is defined as a {\em closed} map $f\colon X\to Y$ satisfying the condition of Definition \ref{c-proper}. However, closedness follows from this condition:  the maps $f_{\vert C_p}\colon C_p\to C_{f(p)}$ are surjective because the fibres are not empty, and hence $f$ is automatically closed.}
\end{rem}

Every closed immersion is c-proper. Every c-proper map is a closed (and universally closed) map. If $f\colon X\to Y$ is c-proper, then, for every base change $\overline Y\to Y$, the induced map $X\times_Y\overline Y\to\overline Y$ is also c-proper. The composition of c-proper maps is c-proper. For every space $X$, the projection to a point, $X\to\{*\}$, is c-proper.

\begin{prop}\label{cohprop-prop} Let $f\colon X\to Y$ be a continuous map. The following conditions are equivalent:
\begin{enumerate} \item $f$ is c-proper.
\item There exists an open covering $Y=V_1\cup\dots\cup V_n$ such that $  f^{-1}(V_i)\to V_i$ is c-proper for every $i=1,\dots,n$.
\item For every $y\in Y$, the map $  f^{-1}(U_y)\to U_y$ is c-proper.
\item There exists a closed covering $Y=C_1\cup\dots\cup C_n$ such that $  f^{-1}(C_i)\to C_i$ is c-proper for every $i=1,\dots,n$.
\item For every $y\in Y$, the map $  f^{-1}(C_y)\to C_y$ is c-proper.
\item There exists a closed covering $X=X_1\cup\dots \cup X_m$ such that $f_{\vert X_i}\colon X_i\to Y$ is c-proper for every $i=1,\dots ,m$.
\item $f_{\vert C_p}\colon C_p\to C_{f(p)}$ is c-proper for every $p\in X$.
\item $f_{\vert C_g}\colon C_g\to C_{f(g)}$ is c-proper for every generic point $g\in X$.
\end{enumerate}
\end{prop}

\begin{thm}\label{cohprop} Let $f\colon X\to Y$ be a continuous map. The following conditions are equivalent:
\begin{enumerate} \item $f$ is c-proper.
\item {\rm (Base change)}. For every cartesian diagram
\[
\xymatrix{ \overline X\ar[d]_{\bar f}\ar[r]^{\bar g} & X \ar[d]^{f} \\   \overline Y\ar[r]^{g}   & \,   Y    
 } 
\] the natural morphism  $g^{-1}\RR f_* F\to  \RR {\bar f}_* {\bar g}^{-1} F$ is an isomorphism  for every $F\in \mathrm{D}(X)$.
\item {\rm ($f^\times$ is local on $Y$)}. For every open subset $V\subseteq Y$ and every $F\in \mathrm{D}(Y)$, the natural morphism $$(f^\times F)_{\vert f^{-1}(V)} \to (f_{\vert f^{-1}(V)})^\times F_{\vert V}  $$ is an isomorphism.
\item{\rm (Local form of duality)}. For every $F\in \mathrm{D}(X)$ and every $K\in \mathrm{D}(Y)$, the natural morphism
\[\RR f_*\RR\HHom_X^\pun(F,f^\times K)\to\RR\HHom_Y^\pun(\RR f_*F,K)\] is an isomorphism.
\end{enumerate}
\end{thm}

The following result will be a key ingredient for our study of proper maps.

\begin{prop}\label{cohpropsup} Let $f\colon X\to Y$ be a c-proper map. For every $p\in X$, let $G_p$ denote   the closure of $\{p\}$ in $f^{-1}(f(p))$ and set $\partial G_p=G_p-\{p\}$. One has a natural isomorphism:
\[\RR f_*\ZZ_{\{p\}} = \ZZ_{\{f(p)\}}\otimes_\ZZ  \RR\Gamma_{\text{\rm red}}(\partial G_p ,\ZZ)[-1].\] 
In particular, $\RR f_*\ZZ_{\{p\}} = \ZZ_{\{f(p)\}}$ if $p$ is $f$-closed (i.e., it is a closed point of the fibre $f^{-1}(f(p))$).
\end{prop}

\begin{proof} For every $y\in Y$, by  base change (Thm. \ref{cohprop}), one has:
\[ (\RR f_*\ZZ_{\{p\}})_y=\RR\Gamma(f^{-1}(y), {\ZZ_{\{p\}}}_{\vert f^{-1}(y)})= 0 \quad \text{if } y\neq f(p).\] Hence, $\RR f_*\ZZ_{\{p\}}$ is supported at $\{f(p)\}$, and its stalk at $f(p)$ is $\RR\Gamma(f^{-1}(f(p)), {\ZZ_{\{p\}}} )$. Thus,
\[ \RR f_*\ZZ_{\{p\}} = \ZZ_{\{f(p)\}}\otimes_\ZZ \RR\Gamma(f^{-1}(f(p)), {\ZZ_{\{p\}}} ).
\]  We conclude the proof because $\RR\Gamma(f^{-1}(f(p)), {\ZZ_{\{p\}}} )= \RR\Gamma_{\text{\rm red}}(\partial G_p  ,\ZZ)[-1]$, by   Lemma \ref{lemma} below, which shall be used several times in the sequel.
\end{proof}

\begin{lem}\label{lemma} Let $X$ be a finite space. For every $p\in X$, one has
\[ \RR\Gamma(X,\ZZ_{\{p\}}) = \RR\Gamma_{\text{\rm red}}(\partial C_p,\ZZ)[-1]\] and for every closed point $x_0< p$:
\[\aligned  \RR\Gamma_{x_0}(X,\ZZ_{\{p\}}) &= \RR\Gamma_{x_0}(\partial C_p,\ZZ)[-1]
\\ \RR\Gamma(X-\{x_0\},\ZZ_{\{p\}}) &= \RR\Gamma_{\text{\rm red}}(\partial C_p-\{x_0\},\ZZ)[-1]
\endaligned\]
\end{lem}

\begin{proof} It suffices to apply $\RR\Gamma(X,-)$, $\RR\Gamma_{x_0}(X,-)$, and $\RR\Gamma(X- \{x_0\},-)$ to the short exact sequence
\[ 0\to\ZZ_{\{p\}}\to \ZZ_{C_p}\to\ZZ_{\partial C_p}\to 0.\] We then take into account, on the one hand, that
\[ \aligned \RR\Gamma(X,\ZZ_{C_p})&= \RR\Gamma(C_p,\ZZ)=\ZZ, 
\\ \RR\Gamma(X-\{x_0\},\ZZ_{C_p})&= \RR\Gamma(C_p-\{x_0\},\ZZ)=\ZZ,
\endaligned\]     because $C_p$ and $ C_p-\{x_0\}$ are contractible to $p$, which implies that
$\RR\Gamma_{x_0}(X ,\ZZ_{C_p})=0$. On the other hand, we have
\[\aligned \RR\Gamma(X,\ZZ_{\partial C_p})&= \RR\Gamma(\partial C_p,\ZZ)
\\ \RR\Gamma_{x_0}(X,\ZZ_{\partial C_p})&= \RR\Gamma_{x_0}(\partial C_p,\ZZ)
\\ \RR\Gamma(X-\{x_0\},\ZZ_{\partial C_p})&= \RR\Gamma(\partial C_p-\{x_0\},\ZZ).\endaligned\]
The desired isomorphisms follow immediately from the long exact sequences associated with these functors.
\end{proof}

\section{  Compatibility between local and global dualizing complexes.}\label{loc-glob-compatibility}
\medskip

Let $f\colon X\to Y$ be a continuous map. The aim of this section is to show how the functor $f^\times$ relates the local dualizing complexes of $X$ and $Y$.

\begin{rem}[\bf Simplified notation]\label{notation}{\rm Let $f\colon X\to Y$ be a continuous map, and let $V\subseteq Y$ be an open subset. In what follows, we shall still denote by $f$ the restricted map $f_{\vert f^{-1}(V)}\colon f^{-1}(V)\to V$, and for every  $F\in \mathrm{D}(V)$, we write $$f^\times F:=(f_{\vert f^{-1}(V)})^\times F.$$ 
Analogously, if $Z$ is a closed subset of $f^{-1}(V)$, we write
\[ f^\times_ZF:= (f_{\vert f^{-1}(V)})^\times_Z F \] for every  $F\in \mathrm{D}(V)$.}
 %where $f_{\vert V}\colon f^{-1}(V)\to V$. 
% Thus, the morphism  \eqref{1} will be written as 
%$$(f^\times D_{U_y}^y)_{\vert U_{x_0}}\to D_{U_{x_0}}^{x_0}.$$}
\end{rem}

For every open subset $j\colon U\hookrightarrow X$, the functor $j_!\colon \mathrm{D}(U)\to \mathrm{D}(X)$ denotes the extension by zero. It is left adjoint to the restriction functor $j^{-1}$. This adjunction can be sheafified to yield a natural isomorphism $\RR j_*\RR\HHom_U^\pun(G,j^{-1}F)=\RR\HHom_X^\pun(j_!G,F)$ for every $G\in \mathrm{D}(U)$ and every $F\in \mathrm{D}(X)$. For every $F\in \mathrm{D}(X)$, we have $F_U=j_!j^{-1}F=F\otimes_\ZZ \ZZ_U$.  

\begin{prop}\label{restringirversuslocal} Let $x_0\in X$ be a closed point and $j\colon U_{x_0}\hookrightarrow X$ be the open immersion. For every $G\in \mathrm{D}(X)$, there is a natural morphism
\[ G_{\vert U_{x_0}}=j^{-1}G\longrightarrow j_{\{x_0\}}^\times G.\]
\end{prop}

\begin{proof} By adjunction, it suffices to define a morphism $ \RR j_* \RR\HHom_{U_{x_0}}^\pun(\ZZ_{\{x_0\}}, j^{-1}G)\to G$. Taking into account that $$ \RR j_* \RR\HHom_{U_{x_0}}^\pun(\ZZ_{\{x_0\}}, j^{-1}G)= \RR\HHom_{X}^\pun(j_!\ZZ_{\{x_0\}},  G) = \RR\HHom_{X}^\pun( \ZZ_{\{x_0\}},  G) ,$$ the desired morphism is obtained directly by applying the functor  $\RR\HHom_{X}^\pun( -,  G)$ to the natural canonical map $\ZZ\to\ZZ_{\{x_0\}}$.

%Para todo $G\in D(U_x)$ el morfismo natural
%\[ j_!\RR\HHom_{U_x}^\pun(\ZZ_{\{x\}}, G) \to \RR j_* \RR\HHom_{U_x}^\pun(\ZZ_{\{x\}}, G)\] es isomorfismo porque 
%$\RR\HHom_{U_x}^\pun(\ZZ_{\{x\}}, G)$ está soportado en $\{x\}$, punto cerrado de $X$. 
\end{proof}

\begin{defn}{\rm Let $f\colon X\to Y$ be a continuous map. We say that a point  $x_0\in X$ is {\em $f$-closed} if $x_0$ is a closed point of the fibre $f^{-1}(f(x_0))$ (equivalently, a closed point of the open subset $f^{-1}(U_{f(x_0)})$). %For each $x\in X$ we shall denote by $X_x$ the fibre of $f$ through $x$, i.e.: $X_x=f^{-1}(f(x))$.
}
\end{defn}

\begin{prop}\label{f-localdual} Let $f\colon X\to Y$ be a continuous map, let  $x_0\in X$ be an $f$-closed point, set $y_0=f(x_0)$, and let $j\colon U_{x_0}\hookrightarrow f^{-1}(U_{y_0})$ be the open immersion. For every $F\in \mathrm{D}(U_{y_0})$,there is a natural morphism (adopting the simplified notation of Remark \ref{notation})
\begin{equation}\label{0} (f^\times_{f^{-1}(y_0)} F)_{\vert U_{x_0}} %= j^{-1} f^\times_{X_{y_0}} F 
\to (j\circ f)^\times_{x_0} F.
\end{equation}
In particular, one has a a natural compatibility morphism:
\begin{equation}\label{1}   ( f^\times D_{U_{y_0}}^{y_0})_{\vert U_{x_0}}\to    D_{U_{x_0}}^{x_0}.\end{equation}
\end{prop}

\begin{proof} By definition, $x_0$ is a closed point of $f^{-1}(U_{y_0})$. Applying Proposition \ref{restringirversuslocal}   to the map $f\colon  f^{-1}(U_{y_0}) \to U_{y_0}$ and setting $G=f^\times_{f^{-1}(y_0)} F$, we obtain a morphism
\[ (f^\times_{f^{-1}(y_0)} F)_{\vert U_{x_0}} \longrightarrow j^\times_{x_0} f^\times_{f^{-1}(y_0)} F = (j\circ f)^\times_{x_0}F, \]
where the last step follows from Proposition \ref{transitivity}.

The specialized morphism \eqref{1} is derived by substituting $F = D_{U_{y_0}}^{y_0}$, since by transitivity \textup{(Proposition \ref{transitivity})} we have:
\[ f^\times_{f^{-1}(y_0)}D_{U_{y_0}}^{y_0} = D_{f^{-1}(U_{y_0})}^{f^{-1}(y_0)} = f^\times D_{U_{y_0}}^{y_0}, \] 
and 
\[ (j\circ f)^\times_{x_0} D_{U_{y_0}}^{y_0} = D_{U_{x_0}}^{x_0}. \]
\end{proof}

In the specific case where $Y$ is a single point, we immediately deduce the following:

\begin{cor}\label{localglobalcompatibility}  For every closed point  $x_0\in X$, there is a natural morphism
$$ (D_X)_{\vert U_{x_0}}\to D_{U_{x_0}}^{x_0}.$$ 

%This morphism is obtained, by adjunction, from the natural morphism $$\RR\Gamma_{x_0}(U_{x_0},-)\to \RR\Gamma(X,j_!(-)),$$ with $j\colon U_{x_0}\hookrightarrow X$.
 
\end{cor}

\begin{defn}\label{proper-2} {\rm Let $f\colon X\to Y$ be a continuous map. We say that {\em $f^\times$ is compatible with local dualizing complexes} if the morphism \eqref{1}
\[ ( f^\times D_{U_{y_0}}^{y_0})_{\vert U_{x_0}}\to    D_{U_{x_0}}^{x_0}\] is an isomorphism for every $f$-closed point $x_0$. If $Y$ is a single point, this condition means that $(D_X)_{\vert U_{x_0}}\to D_{U_{x_0}}^{x_0}$ is an isomorphism for every closed point $x_0\in X$, in which case we say that the global and local dualizing complexes are compatible.}
\end{defn}

\begin{prop}\label{equivalencias} Let $f\colon X\to Y$ be a continuous map, let $x_0\in X$ be an $f$-closed point, and set $y_0=f(x_0)$. The following conditions are equivalent:
\begin{enumerate} 
\item For every $F\in \mathrm{D}(U_{y_0})$,  the morphism \eqref{0} is an isomorphism.
\item The morphism \eqref{1} is an isomorphism.
\item The natural morphism \[\RR\Hom_{X}^\pun(\ZZ_{\{x_0\}}, G)\to \RR\Hom_{Y}^\pun(\ZZ_{\{y_0\}},\RR f_* G)\] is an isomorphism  for every $G\in \mathrm{D}(X)$ supported in $U_{x_0}$.
\item[(3')] The natural morphism \begin{equation}\label{2}\RR\Hom_{f^{-1}(U_{y_0})}^\pun(\ZZ_{\{x_0\}}, G)\to \RR\Hom_{U_{y_0}}^\pun(\ZZ_{\{y_0\}},\RR f_* G)\end{equation} is an isomorphism for every $G\in \mathrm{D}(f^{-1}(U_{y_0}))$ supported in $U_{x_0}$.
\item For every $p\geq x_0$, the natural morphism 
$$\RR\Hom_{X}^\pun(\ZZ_{\{x_0\}}, \ZZ_{\{p\}})\to \RR\Hom_{Y}^\pun(\ZZ_{\{y_0\}},\RR f_* \ZZ_{\{p\}})$$ is an isomorphism.
\item[(4')] For every $p\geq x_0$, the natural morphism 
$$\RR\Hom_{f^{-1}(U_{y_0})}^\pun(\ZZ_{\{x_0\}}, \ZZ_{\{p\}})\to \RR\Hom_{U_{y_0}}^\pun(\ZZ_{\{y_0\}},\RR f_* \ZZ_{\{p\}})$$ is an isomorphism.
\end{enumerate}
\end{prop}

\begin{proof} First, note that any $G\in \mathrm{D}(f^{-1}(U_{y_0}))$ supported in $U_{x_0}$ can be written as $G=j_!M$, where $M\in \mathrm{D}(U_{x_0})$ and $j\colon U_{x_0}\hookrightarrow f^{-1}(U_{y_0})$ is the open immersion. By adjunction, the morphism \eqref{0} defines a natural map
\begin{equation}\label{1'} \RR f_* \RR\HHom_{f^{-1}(U_{x_0})}^\pun(\ZZ_{\{x_0\}}, j_!M)\to \RR\HHom_{ U_{y_0}}^\pun(\ZZ_{\{y_0\}}, \RR f_* j_!M).\end{equation}  Indeed, we have:
\[\aligned \RR\Hom_{U_{x_0}}^\pun(M, j^{-1} f^\times_{f^{-1}(y_0)}F)
&= \RR\Hom_{f^{-1}(U_{x_0})}^\pun(j_!M,   f^\times_{f^{-1}(y_0)}F) \\
&= \RR\Hom_{U_{y_0}}^\pun(\RR f_*\RR\HHom_{f^{-1}(U_{y_0})}^\pun (\ZZ_{f^{-1}(y_0)}, j_!M), F)\\
%(\ZZ_{f^{-1}(y_0)}=f^{-1}\ZZ_{\{y_0\}}) 
&= \RR\Hom_{U_{y_0}}^\pun (\RR\HHom_{U_{y_0}}^\pun(\ZZ_{\{y_0\}}, \RR f_* j_!M), F)
\endaligned\] and
\[ \aligned \RR\Hom_{U_{x_0}}^\pun(M, (j\circ f)^\times_{x_0} F) 
&= \RR\Hom_{U_{y_0}}^\pun (\RR f_*\RR j_*\RR\HHom_{U_{x_0}}^\pun(\ZZ_{\{x_0\}}, M),F)\\
&= \RR\Hom_{U_{y_0}}^\pun (\RR f_* \RR\HHom_{f^{-1}(U_{x_0})}^\pun(\ZZ_{\{x_0\}}, j_!M),F).
\endaligned\] 

Taking global sections in \eqref{1'}, we recover the morphism \eqref{2}. Moreover, \eqref{1'} is an isomorphism if and only if \eqref{2} is an isomorphism, since both sides in \eqref{1'} are supported strictly at $\{y_0\}$. This yields the equivalence of (1) and (3').

The equivalence of (3) and (3') follows from the excision property. Observe that $G\in \mathrm{D}(X)$ is supported in $U_{x_0}$ if and only if $G=(j')_!G'$ for a unique $G'\in \mathrm{D}(f^{-1}(U_{y_0}))$ supported in $U_{x_0}$, where $j'\colon f^{-1}(U_{y_0})\hookrightarrow X$ is the open immersion. Consequently,
\[ \RR\Hom_{X}^\pun(\ZZ_{\{x_0\}}, G)= \RR\Hom_{f^{-1}(U_{y_0})}^\pun(\ZZ_{\{x_0\}}, G')\] and
\[ \RR\Hom_{Y}^\pun(\ZZ_{\{y_0\}},\RR f_* G)= \RR\Hom_{U_{y_0}}^\pun(\ZZ_{\{y_0\}},\RR f_* G')\] by excision. Thus (3) is equivalent to (3'). Analogously, (4) is equivalent to (4').

The equivalence (4')$\Leftrightarrow $(1) follows from the fact that the sheaves  $\{\ZZ_{\{p\}}\}_{p\in U_{x_0}}$ are a system of generators of $\mathrm{D}(U_{x_0})$. Thus, \eqref{0} is an isomorphism for every complex $F$ if and only if it remains an isomorphism after applying the functor  $\RR\Hom_{U_{x_0}}(\ZZ_{\{p\}},-)$ for every $p\in U_{x_0}$. Since
\[\RR\Hom_{U_{x_0}}^\pun(\ZZ_{\{p\}},\eqref{0})=\RR\Hom_{U_{y_0}}^\pun(\eqref{1'},F), \text{ for }M=\ZZ_{\{p\}},\] the assertion holds.

%for any $M\in D(U_{y_0})$. Since both members are supported on $\{y_0\}$, this morphism is an isomorphism if and only if it is an isomorphism after taking global sections, that is, if  (3) holds (notice that $G\in D(f^{-1}(U_{y_0})$ is supported on $U_{x_0}$ iff $G=j_!M$ for  $M\in D(U_{x_0})$).

The implications (3)$\Rightarrow$(4) and (1)$\Rightarrow$(2) are immediate. To conclude, let us see that (2)$\Rightarrow$(4'). This follows because   taking $\RR\Hom_{U_{x_0}}^\pun(\ZZ_{\{p\}},-)$ in the morphism \eqref{1} yields the dual morphism of (4'). Indeed,
 \[ \aligned  \RR\Hom_{U_{x_0}}^\pun(\ZZ_{\{p\}}, (f^\times D_{U_{y_0}}^{y_0})_{\vert U_{x_0}}) &= 
 \RR\Hom_{f^{-1}(U_{y_0})}^\pun(\ZZ_{\{p\}}, f^\times D_{U_{y_0}}^{y_0}) 
 \\ &=  \RR\Hom_{U_{y_0}}^\pun(\RR f_* \ZZ_{\{p\}},  D_{U_{y_0}}^{y_0}) 
 \\ &= \RR\Hom_{U_{y_0}}^\pun (\ZZ_{\{y_0\}}, \RR f_* \ZZ_{\{p\}})^\vee
 %\\ &= \RR\Hom_{Y}^\pun (\ZZ_{\{y_0\}}, \RR f_* \ZZ_{\{p\}})^\vee, \text{ by excision},
 \endaligned\]
 and 
 \[ \RR\Hom_{U_{x_0}}^\pun(\ZZ_{\{p\}}, D_{U_{x_0}}^{x_0}) = \RR\Hom_{U_{x_0}}^\pun (\ZZ_{\{x_0\}}, \ZZ_{\{p\}})^\vee  =  \RR\Hom_{f^{-1}(U_{y_0})}^\pun (\ZZ_{\{x_0\}}, \ZZ_{\{p\}})^\vee, \text{ by excision}.
 \]
\end{proof}

If $Y$ is a single point, Proposition \ref{equivalencias} yields:
\begin{cor}\label{properspace}  For every closed point  $x_0\in X$, the following conditions are equivalent:
\begin{enumerate}
\item The morphism $ (D_X)_{\vert U_{x_0}}\to D_{U_{x_0}}^{x_0}$ is an isomorphism.
\item For every $p\geq x_0$, the natural morphism $\RR\Gamma_{x_0}(X,\ZZ_{\{p\}})\to \RR\Gamma(X,\ZZ_{\{p\}})$ is an isomorphism.
\end{enumerate}

%This morphism is obtained, by adjunction, from the natural morphism $$\RR\Gamma_{x_0}(U_{x_0},-)\to \RR\Gamma(X,j_!(-)),$$ with $j\colon U_{x_0}\hookrightarrow X$.
 
\end{cor}

\begin{rem}\label{immersion} {\rm If $f\colon X\hookrightarrow Y$ is an immersion (i.e., $X$ is a subspace of $Y$), then every point $x_0\in X$ is automatically $f$-closed and
\[ \RR\Hom_X^\pun(G,G')=\RR\Hom_Y^\pun(\RR f_*G, \RR f_*G')\] for all $G,G'\in \mathrm{D}(X)$, since $f^{-1}\RR f_*G=G$ for every $G$. Furthermore, 
$$\RR\Hom_Y(\ZZ_{\{f(x_0)\}}, \RR f_*G)= \RR\Hom_X(f^{-1}\ZZ_{\{f(x_0)\}},  G) = \RR\Hom_X( \ZZ_{\{ x_0\}},  G) ,$$  which implies that  condition (3') (and then all conditions) of Proposition \ref{equivalencias} is satisfied.}
\end{rem}

\section{Proper maps and proper spaces}\label{section-properness}

\begin{defn}\label{defn-proper} {\rm We say that a continuous map $f\colon X\to Y$ is a  {\em proper morphism} if:
\begin{enumerate} 
\item $f$ is c-proper (Definition \ref{c-proper}).
\item  $f^\times$ is compatible with local dualizing complexes (Definition \ref{proper-2}). That is,  any of the equivalent conditions of Proposition \ref{equivalencias} holds for every $f$-closed point.
\end{enumerate} Thus, a proper map $f$ is a c-proper map such that,  for every $f$-closed point $x_0$ and every $p\geq x_0$, the natural morphism
\begin{equation}\label{proper-condition}\RR\Hom_X^\pun(\ZZ_{\{x_0\}} , \ZZ_{\{p\}}) \to \RR\Hom^\pun_Y(\ZZ_{\{f(x_0)\}} , \RR f_*\ZZ_{\{p\}})\end{equation} is an isomorphism. Since this map is automatically an isomorphism for  $p=x_0$ whenever $f$ is c-proper, it suffices to check this condition for all  $p>x_0$.}
\end{defn}

\begin{defn} {\rm A finite space $X$ is called {\em proper}   if the structural projection to a single point, $X\to\{*\}$, is a proper morphism. Since $X\to\{*\}$ is always c-proper, it follows that $X$ is proper if and only if its global and local dualizing complexes are compatible (Definition \ref{proper-2}). By Corollary \ref{properspace}, this is equivalent to saying that, for every closed point  $x_0\in X$ and every $p\geq x_0$, the natural morphism $$\RR\Gamma_{x_0}(X,\ZZ_{\{p\}})\to \RR\Gamma (X,\ZZ_{\{p\}})$$ is an isomorphism. Again,     verifying this property for all $p>x_0$ is sufficient.
 }
\end{defn}

The next theorem characterizes proper spaces in terms of the cohomology of the closed subsets $\partial C_p$.

\begin{thm}\label{thmproperspace} The following conditions are equivalent:
\begin{enumerate}\item $X$ is proper.
\item  For every closed point $x_0\in X$ and every $p>x_0 $, the space $\partial C_p-\{x_0\}$ is cohomologically trivial: 
\[ \RR\Gamma_{\text{\rm red}}( \partial C_p-\{x_0\},\ZZ)=0.\]
\item For every closed point $x_0\in X$ and every $p>x_0 $, one has a natural isomorphism
\[ \RR\Gamma_{x_0}(\partial C_p,\ZZ) = \RR\Gamma_{\text{\rm red}}(\partial C_p,\ZZ).\]
\end{enumerate}   
\end{thm}

\begin{proof} By the triangle of local cohomology, the canonical map  $\RR\Gamma_{x_0}(X,\ZZ_{\{p\}}) \to \RR\Gamma (X,\ZZ_{\{p\}})$ is an isomorphism if and only if   $\RR\Gamma(X-\{x_0\},\ZZ_{\{p\}})=0 $. By Lemma \ref{lemma}, this vanishing  is equivalent to  $\RR\Gamma_{\text{\rm red}}( \partial C_p-\{x_0\},\ZZ)=0$, establishing the equivalence between (1) and (2)..

Similarly, by virtue of Lemma \ref{lemma}, the map $\mathbb{R}\Gamma_{x_0}(X,\mathbb{Z}_{\{p\}}) \to \mathbb{R}\Gamma(X,\mathbb{Z}_{\{p\}})$ is an isomorphism if and only if the map $\mathbb{R}\Gamma_{x_0}(\partial C_p,\mathbb{Z}) \to \mathbb{R}\Gamma_{\text{\rm red}}(\partial C_p,\mathbb{Z})$ is an isomorphism, which proves the equivalence between (1) and (3).
\end{proof}

\begin{exs}{\rm \begin{enumerate}\item An immersion $f\colon X\hookrightarrow Y$ is proper if and only if it is a closed immersion.  Indeed, since any immersion satisfies condition (2) of Definition \ref{defn-proper} (Remark \ref{immersion}), $f$ is proper if and only if $f$ is c-proper, which corresponds precisely to being a closed immersion.
\item Every closed subset of a proper space is  proper.
\item If $f\colon X\to Y$ is a continuous map and $i\colon Y\hookrightarrow Z$ is a closed immersion, 
then $f$ is proper if and only if $i\circ f$ is proper.  Indeed, it is immediate that $f$ is c-proper if and only if $i\circ f$ is c-proper. Furthermore, for any two complexes $F,G\in \mathrm{D}(Y)$, we have $\RR\Hom_Y^\pun(F,G) = \RR\Hom_Z^\pun(i_*F,i_*G)$, because $i$ is a closed immersion. Hence $f$ satisfies condition (4) of Proposition \ref{equivalencias} if and only if $i\circ f$ does.
\end{enumerate}}
\end{exs}

\begin{prop}\label{composition} The composition of proper maps is proper.
\end{prop}

\begin{proof} Let $X\overset f\to Y\overset g\to Z$ be proper maps. The composition $g\circ f\colon X\to Z$ is c-proper, because both $f$ and $g$ are c-proper. To conclude, let $x_0$ be a  $(g\circ f)$-closed point. We must verify that the natural morphism
  \[ \RR\Hom_X^\pun(\ZZ_{\{x_0\}}, G)\to\RR\Hom_Z^\pun (\ZZ_{\{(g\circ f)(x_0))\}},\RR (g\circ f)_*G)\] is an isomorphism for every complex $G\in \mathrm{D}(X)$ supported in $U_{x_0}$.

Let us denote $y_0=f(x_0), z_0=g(f(x_0))$.  First, observe  that $x_0$ is $f$-closed. Since $f$ is proper, we have a natural isomorphism  $$ \RR\Hom_X^\pun(\ZZ_{\{x_0\}}, G)= \RR\Hom_Y^\pun(\ZZ_{\{y_0\}}, \RR f_*G).$$  Moreover,
 $y_0$ is $g$-closed  because $x_0$ is $(g\circ f)$-closed  and $f$ is closed. Additionally,  $\RR f_*G$ is supported in $U_{y_0}$: indeed, for every $y\in Y$, Theorem \ref{cohprop} implies
\[ (\RR f_* G)_y = \RR\Gamma (f^{-1}(y),G_{\vert f^{-1}(y)})\] which vanishes whenever $y\notin U_{y_0}$, since  $\supp (G)\subseteq U_{x_0}\subseteq f^{-1}(U_{y_0})$.
Thus,  since $g$ is proper, we obtain: 
\[ \RR\Hom_Y^\pun(\ZZ_{\{y_0\}}, \RR f_*G)=\RR\Hom_Z^\pun (\ZZ_{\{z_0\}},\RR g_* \RR f_*G).\] Combining these isomorphisms yields   $ \RR\Hom_X^\pun(\ZZ_{\{x_0\}}, G)=\RR\Hom_Z^\pun (\ZZ_{\{z_0\}},\RR (g\circ f)_*G)$, as required.

\end{proof}

The following proposition is straightforward  and its verification  is left to the reader.

\begin{prop} Let $f\colon X\to Y$ be a continuous map. The following conditions are equivalent:
\begin{enumerate} 
\item $f$ is proper.
\item There exists an open covering $Y=V_1\cup\dots\cup V_n$ such that the restriction $f\colon f^{-1}(V_i)\to V_i$ is proper for every $i=1,\dots,n$.
\item For every $y\in Y$, the map $f\colon f^{-1}(U_y)\to U_y$ is proper.
\item There exists a closed covering $Y=C_1\cup\dots\cup C_n$ such that $f\colon f^{-1}(C_i)\to C_i$ is proper for every $i=1,\dots,n$.
\item For every $y\in Y$, the map $f\colon f^{-1}(C_y)\to C_y$ is proper.
\item There exists a closed covering $X=X_1\cup\dots \cup X_m$ such that $f_{\vert X_i}\colon X_i\to Y$ is proper for every $i=1,\dots ,m$.
\item $f_{\vert C_p}\colon C_p\to C_{f(p)}$ is proper for every $p\in X$.
\item $f_{\vert C_g}\colon C_g\to C_{f(g)}$ is proper for every generic point $g\in X$.
\end{enumerate}
\end{prop}

\begin{cor} The following conditions are equivalent:
\begin{enumerate} \item $X$ is proper.
\item There exists a closed covering $X=X_1\cup\dots \cup X_m$ such that $ X_i$ is proper for every $i=1,\dots ,m$.
\item $ C_p$ is proper for every $p\in X$.
\item Every irreducible component of $X$ is proper.
\end{enumerate}
\end{cor}

\begin{prop}\label{basechange0} Let $f\colon X\to Y$ be a proper map. For every finite space $T$, the induced map  $f\times1\colon X\times T  \to   Y\times T $ is proper. Consequently, the product of two proper spaces is proper.
\end{prop}

\begin{proof} The map $f\times 1$ is c-proper, because $f$ is c-proper. To conclude, let $z_0=(x_0,t_0)$ be an $(f\times 1)$-closed point, and let $z=(x,t)$ be a point such that $z \geq z_0$. We must check that (setting $y_0=f(x_0)$) the canonical map
$$\RR\Hom_{X\times T}^\pun(\ZZ_{\{(x_0,t_0)\}}, \ZZ_{\{(x,t)\}})\to \RR\Hom_{Y\times T}^\pun(\ZZ_{\{(y_0,t_0)\}},\RR (f\times 1)_* \ZZ_{\{(x,t)\}})$$ is an isomorphism. For any two complexes $F\in \mathrm{D}(X),G\in \mathrm{D}(T)$, we denote their external tensor product by
\[ F\boxtimes G:=\pi_X^{-1}F\overset\LL\otimes\pi_T^{-1}G\] where $\pi_X\colon X\times T\to X$ and $\pi_T\colon X\times T\to T$ are the natural projections. We adopt the same notation  $F\boxtimes G$ for $F\in \mathrm{D}(Y)$ and $G\in \mathrm{D}(T)$. One has:
\[ \ZZ_{\{(x_0,t_0)\}} = \ZZ_{\{x_0\}}\boxtimes \ZZ_{\{t_0\}} \quad,\quad \ZZ_{\{(x,t)\}} = \ZZ_{\{x\}}\boxtimes \ZZ_{\{t\}}\] and
\[ \RR (f\times 1)_* \ZZ_{\{(x,t)\}}= \RR (f\times 1)_* (\ZZ_{\{x\}}\boxtimes \ZZ_{\{t\}}) =  \RR f_* \ZZ_{\{x\}}\boxtimes \ZZ_{\{t\}}.\]
 Then applying the results from \cite[Prop. 4.23]{ST2} for the external Hom expansions, we obtain:
\[\aligned  \RR\Hom_{X\times T}^\pun(\ZZ_{\{x_0\}}\boxtimes \ZZ_{\{t_0\}},  \ZZ_{\{x\}}\boxtimes \ZZ_{\{t\}}) &= \RR\Hom_{X}^\pun(\ZZ_{\{x_0\}},  \ZZ_{\{x\}})\overset\LL\otimes
\RR\Hom_{ T}^\pun(\ZZ_{\{t_0\}},   \ZZ_{\{t\}})
\\  (f\text{ is proper})& = \RR\Hom_{Y}^\pun(\ZZ_{\{y_0\}},  \RR f_*  \ZZ_{\{x\}})\overset\LL\otimes
\RR\Hom_{ T}^\pun(\ZZ_{\{t_0\}},   \ZZ_{\{t\}})
\\ & = \RR\Hom_{Y\times T}^\pun(\ZZ_{\{y_0\}}\boxtimes \ZZ_{\{t_0\}},\RR f_* \ZZ_{\{x\}}\boxtimes \ZZ_{\{t\}})
\\ & = \RR\Hom_{Y\times T}^\pun(\ZZ_{\{y_0\}}\boxtimes \ZZ_{\{t_0\}}, \RR (f\times 1)_* (\ZZ_{\{x\}}\boxtimes \ZZ_{\{t\}})
. \endaligned\] 
For the consequence, let $X,Y$ be proper spaces. Since $X$ is proper, the map   $X\times Y\to Y$ is proper by Proposition \ref{basechange0}. Since $Y$ is proper, the map $Y\to\{*\}$ is proper, and therefore the composition $X\times Y\to Y\to \{*\}$ is also proper by Proposition \ref{composition}, which means that the product space $X\times Y$ is proper. 
\end{proof}

\begin{cor}\label{properbasechange} Let $f\colon X\to Y$ be a proper map. For every base change $\overline Y\to Y$ locally of finite type, the induced map $X\times_Y\overline Y \to \overline Y $ is proper.
\end{cor}

\begin{proof} We are reduced to the case where $  \overline Y$ and $ Y$ are local and $  \overline Y\to Y$ is a local morphism (i.e., it maps the unique closed point to unique the closed point). In this case, $\overline Y\to Y$ factors as a closed immersion $\overline Y\to Y\times L$   followed by the natural projection  $Y\times L\to Y$. Since both types of base changes (closed immersions and projections) preserve properness, the assertion holds.
\end{proof}

\begin{thm}\label{thmproperonfibres} Let $f\colon X\to Y$ be a c-proper map. For each point $p\in X$, let    $X_p=f^{-1}f(p))$ denote the fibre of $f$ through $p$. %$G_p=f_p^{-1}(f(p))$, $\partial G_p=G_p-\{p\}$.  
The following conditions are equivalent: 
\begin{enumerate}
\item $f$ is proper.
\item $f$ has proper fibres, and for every $f$-closed point $x_0$ and every $p>x_0$ such that $f(p)>f(x_0)$, one has  
\[ \RR\Gamma_{\text{\rm red}}(\,(x_0,p),\ZZ)= \RR\Gamma_{\text{\rm red}}(\, (f(x_0),f(p)),\ZZ)\overset\LL\otimes  \RR\Gamma_{\text{\rm red}}(\partial G_p,\ZZ)[-1],\] where
\[ G_p=\text{\rm closure of } \{p\} \text{ \rm in } X_p\quad , \quad \partial G_p =G_p-\{p\}.\] 
\item $f$ has proper fibres, and for every $f$-closed point $x_0$ and every $p>x_0$ such that $f(p)>f(x_0)$, one has
\[\RR\Gamma_{(X_{x_0}\cap\partial C_p^{f(x_0)})-\{ x_0\}} (\partial C_p^{f(x_0)} -\{x_0\},\ZZ)=0,\] 
where
\[\aligned C_p^{f(x_0)}& =\text{ \rm closure of } \{p\} \text{ \rm in } f^{-1}(U_{f(x_0)}) =C_p\cap f^{-1}(U_{f(x_0)}) \\ \partial C_p^{f(x_0)} &=C_p^{f(x_0)}-\{p\}.\endaligned\] 
\end{enumerate}
\end{thm}

\begin{proof} (1)$\Rightarrow$(2). Let us assume that  $f$ is proper. The fibres of $f$ are proper by Corollary  \ref{properbasechange}, because the base change $\{y\}\hookrightarrow Y$ is locally of finite type (it can be written as the composition of the closed immersion $\{y\}\hookrightarrow U_{y}$ followed by the open immersion $U_y\hookrightarrow Y$). Now,  we establish the second condition of (2).

Since $f$ is proper, the map
\[\RR\Hom_X^\pun(\ZZ_{\{x_0\}} , \ZZ_{\{p\}}) \overset{\eqref{proper-condition}}\longrightarrow \RR\Hom^\pun_Y(\ZZ_{\{f(x_0)\}} , \RR f_*\ZZ_{\{p\}})\] is an isomorphism for every $f$-closed point $x_0$ and every $p> x_0$. Moreover, by \cite[Prop. 2.5]{ST2}, we have
\begin{equation} \label{interval}  
\RR\Hom_X^\pun(\ZZ_{\{x_0\}} , \ZZ_{\{p\}})=\RR\Gamma_{\text{\rm red}}((x_0,p),\ZZ)[-2].
\end{equation}

 Now, if $f(p)>f(x_0)$, we obtain the following chain of isomorphisms:

\[\aligned \RR\Gamma_{\text{\rm red}}(\,(x_0,p),\ZZ)[-2]& \overset{\eqref{interval}}= \RR\Hom_{X}^\pun(\ZZ_{\{x_0\}} , \ZZ_{\{p\}}) \overset{\eqref{proper-condition}}= \RR\Hom^\pun_Y(\ZZ_{\{f(x_0))\}} , \RR f_*\ZZ_{\{p\}})
\\ &\overset{\ref{cohpropsup}}= \RR\Hom^\pun_Y(\ZZ_{\{y_0)\}} ,  \ZZ_{\{f(p)\}}) \overset\LL\otimes \RR\Gamma_{\text{\rm red}}(\partial G_p,\ZZ)[-1]\\ & \overset{\eqref{interval}}= \RR\Gamma_{\text{\rm red}}(\,(f(x_0),f(p)),\ZZ)[-2] \overset\LL\otimes \RR\Gamma_{\text{\rm red}}(\partial G_p,\ZZ)[-1]\endaligned
\] and shifting the complexes yields the desired relation.

Reversing the previous arguments,  we obtain that (2) implies that \eqref{proper-condition}  is an isomorphism, whenever   $f(p)>f(x_0)$. Next, let us verify that \eqref{proper-condition} remains an isomorphism when $f(p) = f(x_0)$.   In this setting, $x_0\in X_p$, and the result follows from the properness of the fibre $X_p$. Indeed, on the one hand,
\[ \aligned \RR\Hom_{X}^\pun(\ZZ_{\{x_0\}} , \ZZ_{\{p\}})&\overset{\eqref{interval}}=  \RR\Hom_{X_p}^\pun(\ZZ_{\{x_0\}} , \ZZ_{\{p\}})=\RR\Gamma_{x_0}(X_p,\ZZ_{\{p\}})\\ & \overset{\ref{lemma}}=\RR\Gamma_{x_0}(\partial G_p,\ZZ)[-1] \overset{\ref{thmproperspace} }=\RR\Gamma_{\text{\rm red}}(\partial G_p,\ZZ)[-1].\endaligned\]
On the other hand,
\[ \aligned \RR\Hom_Y^\pun(\ZZ_{\{f(x_0)\}},\RR f_*\ZZ_{\{p\}}) &\overset{\ref{cohpropsup} }=\RR\Hom_Y^\pun(\ZZ_{\{f(p)\}},\ZZ_{\{f(p)\}})\overset\LL\otimes \RR\Gamma_{\text{\rm red}}(\partial G_p,\ZZ)[-1]\\ &=\ZZ\overset\LL \otimes \RR\Gamma_{\text{\rm red}}(\partial G_p,\ZZ)[-1]=  \RR\Gamma_{\text{\rm red}}(\partial G_p,\ZZ)[-1] \endaligned\] which matches the first expression.

(1)$\Rightarrow$(3). We already know that $f$ has proper fibres, so it remains to show the second condition of (3).  By hypothesis, setting $y_0 = f(x_0)$, we have:
\[ \RR\Hom_X^\pun(\ZZ_{\{x_0\}} , \ZZ_{\{p\}}) \overset{\eqref{proper-condition}}= \RR\Hom^\pun_Y(\ZZ_{\{y_0\}} , \RR f_*\ZZ_{\{p\}}) \overset{\text{\rm adjunction}}= \RR\Hom_X^\pun (\ZZ_{X_{x_0}},\ZZ_{\{p\}}).\]  By excision, this induces an isomorphism
\[ \RR\Hom_{f^{-1}(U_{y_0})}^\pun(\ZZ_{\{x_0\}} , \ZZ_{\{p\}}) = \RR\Hom_{f^{-1}(U_{y_0})}^\pun (\ZZ_{X_{x_0}},\ZZ_{\{p\}})\] i.e., an isomorphism
\[ \RR\Gamma_{x_0}(f^{-1}(U_{y_0}), \ZZ_{\{p\}})= \RR\Gamma_{X_{x_0}}(f^{-1}(U_{y_0}), \ZZ_{\{p\}})\] or equivalently, the vanishing condition
\[ \RR\Gamma_{X_{x_0}-\{x_0\}}(f^{-1}(U_{y_0})-\{x_0\}, \ZZ_{\{p\}}) =0.\]
We conclude by noting  the isomorphism
\[ \RR\Gamma_{X_{x_0}-\{x_0\}}(f^{-1}(U_{y_0})-\{x_0\}, \ZZ_{\{p\}}) = \RR\Gamma_{(X_{x_0}\cap\partial C_p^{f(x_0)})-\{ x_0\}} (\partial C_p^{f(x_0)} -\{x_0\},\ZZ)\] which is proved in the same way as Lemma \ref{lemma}.

Finally, reversing these arguments, we deduce that (3) implies that \eqref{proper-condition} is an isomorphism whenever $f(p)>f(x_0)$. Moreover,  it is also an isomorphism when $f(p)=f(x_0)$ because the fibre $X_p$ is proper.
\end{proof}

\subsection{Properness on locally dualizable spaces.}\label{subsection-ld-properness}
\bigskip

The results concerning proper maps and spaces can be  significantly refined  under the locally dualizable hypothesis, due to the following facts on locally dualizable spaces:
\begin{enumerate} \item[(A)] For every $p\in X$  and every closed point $x_0\in\partial C_p$, one has
\[\RR\Gamma_{x_0}(\partial C_p,\ZZ)=\ZZ[-\codim (x_0,\partial C_p)].\]
\item[(B)] For every pair of points  $x<p$, the closed interval $[x,p]$ is pure and the open interval  $(x,p)$ is a cohomological sphere; that is \[\RR\Gamma_{\text{\rm red}}((x,p),\ZZ) = \ZZ[-\dim (x,p)].\]
\end{enumerate}

\begin{lem}\label{pure} Let $X$ be an irreducible dualizable finite space. The following conditions are equivalent:
\begin{enumerate}
\item $X$ is pure.
\item $\partial X$ is pure.
\item Every closed point of $X$ has codimension equal to $\dim X$.
\end{enumerate}
\end{lem}

\begin{proof} The equivalence between (1) and (2) is immediate, and the implication (1) $\Rightarrow $ (3) is also direct. Finally, (3) implies (1) because every the maximal chain starting at a closed point $x_0$ and ending at the generic point $g$ of $X$  have the exact same length, since the closed interval $[x_0,g]$ is pure by hypothesis.
\end{proof}

\begin{prop}\label{prop+irred} Let $X$ be a proper, irreducible, and dualizable space. Then $\partial X$ is pure and a cohomological sphere. 
\end{prop}

\begin{proof} Since $X$ is proper, for every closed point $x_0\in\partial X$, there is a natural isomorphism
\[\RR\Gamma_{x_0}(\partial X,\ZZ)=\RR\Gamma_{\text{\rm red}}(\partial X,\ZZ).\]
Moreover, by property (A), we have $\RR\Gamma_{x_0}(\partial X,\ZZ)=\ZZ[-\codim (x_0,\partial X)]$. Combining these facts, we obtain:
\[ \ZZ[-\codim (x_0,\partial X)] = \RR\Gamma_{\text{\rm red}}(\partial X,\ZZ).\] This implies that the value of  $\codim (x_0,\partial X)$ is independent of the choice of $x_0$, which means that  $\codim (x_0,\partial X)=\dim\partial X$. Hence, $\partial X$ is pure (Lemma \ref{pure})   and a cohomological sphere.
\end{proof}

\begin{cor}\label{pro+irred2} Let $X$ be a proper and locally dualizable space. For every $p\in X$, the space $\partial C_p$ is pure and a cohomological sphere.
\end{cor}

\begin{proof} Since $X$ is proper and locally dualizable, the closed subset  $C_p$ is proper and dualizable. The assertion then follows immediately from Proposition \ref{prop+irred}. 
\end{proof}

The next proposition shows that proper maps satisfy a compatibility property with respect to canonical complexes that is highly analogous to the behavior of proper morphisms of schemes.

\begin{prop}\label{proper-canonical} Let $f\colon X\to Y$ be a c-proper morphism between locally dualizable spaces. The folowing conditions are equivalent (adopting the abbreviated notation established in Remark\ref{notation}).
\begin{enumerate}
\item $f$ is proper.
\item For every closed point $y_0\in Y$, $f^\times \Omega$ is a canonical complex on $f^{-1}(U_{y_0})$, where $\Omega$ is a canonical complex on  $U_{y_0}$. 
\item For every $y\in Y$, $ f^\times \Omega$ is a canonical complex on $f^{-1}(U_y)$, where $\Omega$ is a canonical complex on  $U_y$. 

\end{enumerate}
\end{prop}

\begin{proof} First, recall that on a locally dualizable space, for every $x\in X$, the local dualizing complex $D_{U_x}^x$ is a canonical complex on $U_x$  and it is unique up to shift.

(1) $\Rightarrow$ (2). Let $\Omega $ be a canonical complex on $U_{y_0}$. Since $U_{y_0}$ is local, we may asssume without loss of generality that $\Omega=D_{U_{y_0}}^{y_0}$. Since $f$ is proper, for every closed point $x_0\in f^{-1}(U_{y_0})$, one has $ (f^\times \Omega)_{\vert U_{x_0}}=D_{U_{x_0}}^{x_0}$, which is a canonical complex on $U_{x_0}$  because  $X$ is locally dualizable.  Since the open set $f^{-1}(U_{y_0})$ is covered by these open subsets $U_{x_0}$, it follows that $f^\times \Omega$ is indeed a canonical complex on $f^{-1}(U_{y_0})$.

(2) $\Rightarrow$ (3). For every $y\in Y$, let $y_0\leq y$ be a closed point of $Y$ and denote: $$\Sigma_y=f^\times D_{U_y}^y\quad ,\quad \Sigma_{y_0}=f^\times D_{U_{y_0}}^{y_0}.$$   By \cite[Thm. 4.16]{ST2}, we have 
\[(D_{U_{y_0}}^{y_0})_{\vert U_y}=D_{U_y}^y[r],\quad r=\dim [y_0,y]\] and then, since $f^\times $ is local on $Y$ (Thm. \ref{cohprop}), we get: 
\[ (\Sigma_{y_0})_{\vert f^{-1}(U_y)}= \Sigma_y[r].\] Let $x\in f^{-1}(U_y)$ be a closed point and let $x_0\leq x$ be a closed point of $f^{-1}(U_{y_0})$. By the hypothesis in (2), we know that $(\Sigma_{y_0})_{\vert U_{x_0}}=D_{U_{x_0}}^{x_0}$. Therefore,
\[(\Sigma_y)_{\vert U_x}=(\Sigma_{y_0}[-r])_{\vert U_x} = (D_{U_{x_0}}^{x_0}[-r])_{\vert U_x} = D_{U_x}^x[s-r]\] where $s=\dim[x_0,x]$, and the last isomorphism is again due to \cite[Thm. 4.16]{ST2}. Thus, $\Sigma_y$ is a canonical complex on $f^{-1}(U_y)$, because its restriction to every $U_x$ (where $x$ ranges over the  closed points of $f^{-1}(U_y)$) is canonical.

To conclude, we establish  (3) $\Rightarrow$ (1).  By hypothesis, $\Sigma=f^\times D_{U_y}^y$ is a canonical complex on $f^{-1}(U_y)$, which implies that $\Sigma_{\vert U_{x_0}}$ is a canonical complex on $U_{x_0}$. The morphism $\Sigma_{\vert U_{x_0}}\to D_{U_{x_0}}^{x_0}$ is an isomorphism because it becomes one  after applying functor  $\RR\Hom^\pun_{U_{x_0}}(\ZZ_{\{x_0\}},-)$ (see Remark \ref{remark}).
\end{proof}

\begin{cor}\label{dualizante=canonico} Let  $X$ be a locally dualizable space. Then, $X$ is proper if and only if the global dualizing complex  $D_X$ is a canonical complex (which in turn implies that  $X$ is dualizable).
\end{cor}

\begin{cor}\label{c-proper=proper} Let $X$ and $Y$ be proper and dualizable. Every c-proper map $f\colon X\to Y$ is proper.
\end{cor}
\begin{proof} Since $Y$ is proper and dualizable, $D_Y$ is a canonical complex on $Y$.  Because $X$ is also proper and dualizable, $f^\times D_Y=D_X$ is a canonical complex on $X$. The result follows directly from Proposition \ref{proper-canonical}.
\end{proof}

The next result illustrates  the abundance of proper subsets within any locally dualizable space.

\begin{cor}\label{abundance} Let $X$ be a local space  with closed point $\0$. If $X$ is dualizable, then $X^*$ is proper. Consequently, for every locally dualizable space $X$ and every point $p\in X$, the open subspace $U_p^*$ is proper.
\end{cor}

\begin{proof} Let $j\colon X^*\hookrightarrow X$ be the open immersion. From the exact triangle (see \cite[(3.6.2)]{ST2})
\[ \RR j_*D_{X^*}\to\ZZ_{\{\0\}}\to D_X^\0\] we obtain by restriction that $D_{X^*}= (D_X^\0)_{\vert X^*}[-1].$ Hence, $D_{X^*}$ is a canonical complex on $X^*$, and the claim follows from Corollary \ref{dualizante=canonico}.
\end{proof}

The following theorem gives a cohomological characterization of proper and dualizable spaces.

\begin{thm}\label{proper+dualizable} Let $X$ be a locally dualizable space. The following conditions are equivalent:
\begin{enumerate}
\item $X$ is proper.
\item For every $p\in X$,  the space  $\partial C_p$ is pure and  a cohomological sphere.
\item For every $p\in X$, the space $\partial C_p $ is pure and $\partial C_p-\{x_0\}$ is cohomologically trivial for some closed point $x_0\in\partial C_p$.
\end{enumerate}
\end{thm}

\begin{proof} 

The implication (1) $\Rightarrow $ (2) is Corollary \ref{pro+irred2}. For (2) $\Rightarrow $ (1), we proceed by induction on $\dim X$. If $\dim X=0$, then $X$ is trivially proper. Let us assume $\dim X >0$. We may assume that $X$ is irreducible with generic point  $g$,    and let $j\colon \{g\}\hookrightarrow X$ denote the open immersion. If $\dim X=1$, 
then $\partial X$ is a finite discrete set of closed points. More precisely, $\partial X$ consists of exactly two closed points because $\partial X$ is a cohomological sphere by assumption. This implies that $X=\PP_{\FF_1}^1$, which is proper   by Theorem \ref{thmproperspace}. Thus, we may assume henceforth that $\dim X>1$. 

We know that $\ZZ_{\{g\}}$ is a canonical complex on $X$. By Corollary \ref{dualizante=canonico}, it suffices to show that $D_X=\ZZ_{\{g\}}[n]$, where $n=\dim X$. Let us consider the exact triangle
\[ \RR\HHom_X^\pun(\ZZ_{\partial X},D_X)\to D_X\to \RR j_* (D_X)_g.\]

Let us compute the stalk $(D_X)_g$. One has
\[ (D_X)_g=\RR\Hom_X^\pun(\ZZ_{\{g\}},D_X)=\RR\Gamma(X,\ZZ_{\{g\}})^\vee \overset{\ref{lemma}}= \RR\Gamma_{\text{\rm red}}(\partial X,\ZZ)^\vee[1] = \ZZ[n]\] where the last isomorphism holds because $\partial X$ is assumed to be a cohomological sphere. Consequently, $$\RR j_*(D_X)_g=\ZZ[n].$$ 

Now, let us compute $\RR\HHom_X^\pun(\ZZ_{\partial X},D_X)$. If $i\colon\partial X\hookrightarrow X$ denotes the closed immersion, then (see, for example, \cite[Thm. 3.5]{ST2}) we have $$\RR\HHom_X^\pun(\ZZ_{\partial X},D_X)=i_* D_{\partial X}.$$ By the induction hypothesis, $\partial X$ is proper. Since $\ZZ$ is a canonical complex on $\partial X$ (Proposition \ref{partial X}), one has $D_{\partial X}=\Lc[r]$ for some invertible sheaf $\Lc$ on $\partial X$ and some integer $r$. Let us show that $r=n-1$ and $\Lc=\ZZ$.

Let $x_0\in\partial X$ be a closed point. Then we have
\[ \ZZ=\RR\Hom_{\partial X}^\pun(\ZZ_{\{x_0\}},D_{\partial X})\overset{\Lc_{\vert U_{x_0}}=\ZZ}{===} \RR\Hom_{U_{x_0}}^\pun( \ZZ_{\{x_0\}},\ZZ[r]) =   
\ZZ[-\codim(x_0,\partial X) +r]\] and then $r=\codim(x_0,\partial X)=n-1$, since $\partial X$ is pure by hypothesis.

In order to see that $\Lc=\ZZ$, it suffices to prove that $H^0(X,\Lc)\neq 0$ (note that $\partial X$ is connected because it is a cohomological sphere  of dimension $>0$). Indeed, we have
\[H^0(\partial X,\Lc)=H^{-n+1}(\partial X,\Lc[n-1])= H^{-n+1}(\partial X,D_{\partial X}) = H^{n-1}(\partial X,\ZZ)^*=\ZZ\] where the last isomorphism holds because $\partial X$ is a cohomological sphere. 

Combining these facts, the initial exact triangle becomes
\[ \ZZ_{\partial X}[n-1]\to D_X\to \ZZ[n].\] Thus, we obtain a distinguished exact triangle   $$D_X\to \ZZ[n]\overset{\phi}\to \ZZ_{\partial X}[n].$$ 
Taking cohomology, we extract the exact sequence 
\[\ZZ=H^0(X,\ZZ)\overset{H^{-n}(X,\phi)}\longrightarrow \ZZ=H^0(X,\ZZ_{\partial X})\to H^{-n+1}(X,D_X) .\] Since $X$ is contratible to $g$, $H^{-n+1}(X,D_X)=0$. Hence, $\phi[-n]\colon\ZZ\to\ZZ_{\partial X}$ is the natural epimorphism (up to a sign). In conclusion, we obtain
\[D_X=\ZZ_{\{g\}}[n]\] as desired.

The implication (1) $\Rightarrow $ (3) is immediate from Theorem \ref{thmproperspace}. Finally, let us show that (3) $\Rightarrow $ (2). Let $x_0$ be a closed point such that $\partial C_p-\{x_0\}$ is cohomologically trivial. Since $\partial C_p$ is pure, we have $\codim(x_0,\partial C_p)=\dim\partial C_p$. From the exact triangle of local cohomology, we obtain:  \[ \ZZ [-\dim \partial C_p ] =\RR\Gamma_{x_0}(\partial C_p,\ZZ)\to \RR\Gamma (\partial C_p,\ZZ)\to \RR\Gamma (\partial C_p-\{ x_0\},\ZZ)=\ZZ \]
and then $\partial C_p$ is a cohomological sphere. 
\end{proof}

%\begin{rem} {\rm As  has been shown in the proof, if $X$ is proper, then $\partial C_p-\{x_0\}$ is cohomologically trivial for \emph{any} closed point $x_0\in\partial C_p$.}
%\end{rem}

If $X$ is proper and dualizable, then for every point \(p\), the space $\partial C_p$ is not only a cohomological sphere, but its dualizing complex is also that of a sphere:

\begin{prop} Let $X$ be a proper and dualizable space. For every $p\in X$, one has

\begin{enumerate}
\item The dualizing complex $D_{C_p}$ is a canonical complex on $C_p$, and \[D_{C_p}=\ZZ_{\{p\}}[\dim C_p].\]

\item The dualizing complex $D_{\partial C_p}$ is a canonical complex on $\partial C_p$, and
\[ D_{\partial C_p}=\ZZ [\dim \partial C_p].\]
\end{enumerate}
\end{prop}

\begin{proof} (1) Since $C_p$ is proper and dualizable, its dualizing complex is canonical. Moreover, as $C_p$ is   irreducible with generic point $p$, we have $D_{C_p}=\ZZ_{\{p\}}[r]$ for a unique integer $r$. Taking $\RR\Hom_{C_p}(\ZZ_{\{p\}},-)$ we obtain
\[ \RR\Gamma(C_p,\ZZ_{\{p\}})^\vee=\ZZ[r].\] Since $\RR\Gamma(C_p,\ZZ_{\{p\}})\overset{\ref{lemma}}=\RR\Gamma_{\text{\rm red}}(\partial C_p,\ZZ)[-1]$ and $\partial C_p$ is a cohomological sphere, we conclude that $r=\dim C_p$.

(2) Since $\partial C_p$ is proper and dualizable, its dualizing complex is canonical. The formula follows from (1) by applying $i^{-1}\RR\HHom_{C_p}^\pun(\ZZ_{\partial C_p},-)$, where $i\colon\partial C_p\to C_p$ is the closed immersion.
\end{proof}

Let us now see what Theorem \ref{thmproperonfibres} looks like under the locally dualizable hypothesis.

\begin{thm}\label{properfibresld} Let $f\colon X\to Y$ be a c-proper map between locally dualizable spaces. The following conditions are equivalent:
\begin{enumerate}\item $f$ is proper.
\item $f$ has proper fibres and, for every $f$-closed point $x_0$ and every $p>x_0$ such that $f(p)> f(x_0)$, one has: 
\[\dim [x_0,p]=\dim [f(x_0),f(p)] + \dim G_p 
\] where $G_p$  is the closure of $\{p\}$ in $f^{-1}(f(p))$, i.e.,  the generic fibre of $f_{\vert C_p}\colon C_p\to C_{f(p)}$.
\end{enumerate}
\end{thm}

\begin{proof} This is a direct  consequence of the equivalence between (1) and (2) of Theorem \ref{thmproperonfibres}, taking into account  that, under the local dualizable hypothesis,   property (B) implies: 
\[ \aligned \RR\Gamma_{\text{\rm red}}((x_0,p),\ZZ)&=\ZZ[2-\dim [x_0,p]]\\  \RR\Gamma_{\text{\rm red}}((f(x_0),f(p)),\ZZ)&=\ZZ[2-\dim [f(x_0),f(p)]]\endaligned\]
and $\RR\Gamma_{\text{\rm red}}(\partial G_p,\ZZ)=\ZZ[-\dim \partial G_p]=\ZZ[1-\dim G_p]$ whenever the fibre $f^{-1}(f(p))$ is proper (Theorem \ref{proper+dualizable})

\end{proof}

\begin{cor}\label{properonvarieties} Let $f\colon X\to Y$ be a proper morphism between locally dualizable spaces. For every $p\in X$ one has:\[ \dim C_p=\dim C_{f(p)}+\dim G_p.\] In particular, $\dim C_p=\dim C_{f(p)}$ for every $f$-closed point $p$. 
\end{cor}

\begin{proof} Let $x_0\in C_p$ be a closed point such that $\dim [x_0,p]=\dim C_p$. By Theorem \ref{properfibresld}, we have
\[ \dim C_p =\dim [f(x_0),f(p)] +\dim G_p \leq \dim C_{f(p)}+\dim G_p.\]
Now, let $y_0\in C_{f(p)}$ be a closed point such that $\dim [y_0,f(p)]=\dim C_{f(p)}$ and choose an $f$-closed point $x_0\in C_p\cap f^{-1}(y_0)$. Applying Theorem \ref{properfibresld} again yields
\[ \dim C_p\geq \dim [x_0,p]=\dim [f(x_0),f(p)] +\dim G_p  = \dim C_{f(p)}+\dim G_p.\]

Finally, if $p$ is $f$-closed, then $G_p=\{p\}$ and $\dim G_p=0$.
\end{proof}

This result admits a converse if we assume $X$ and $Y$ to be pure:

\begin{prop}\label{proper-formula} Let $f\colon X\to Y$ be a c-proper map between locally dualizable and pure spaces (more generally, assume that $C_p$ and $C_{f(p)}$ are pure for every $p\in X$). Then $f$ is proper if and only if it has proper fibres and
\[ \dim C_p=\dim C_{f(p)}+\dim G_p \] for every $p\in X$.
\end{prop}
\begin{proof} If $f$ is proper,  it has proper fibres and the dimension formula holds by Corollay  \ref{properonvarieties}. For the converse, it suffices to check that $f_{\vert C_p}\colon C_p\to C_{f(p)}$ is proper, so we may assume without loss of generality that  $X$ and $Y$ are pure. By Theorem  \ref{properfibresld}, it is enough to show that, for every    $f$-closed point $x_0$ and every $p>x_0$ such that $f(p)>f(x_0)$, one has 
\[\dim [x_0,p]=\dim [f(x_0),f(p)] + \dim G_p. \]  
Since $X$ and $Y$ are pure, we have
\[ \dim [x_0,p]=\dim C_p-\dim C_{x_0}\quad ,\quad \dim[f(x_0),f(p)]=\dim C_{f(p)}-\dim C_{f(x_0)}.\]
By hypothesis, the relations
\[ \dim C_p = \dim C_{f(p)} + \dim G_p \quad \text{and} \quad \dim C_{x_0} = \dim C_{f(x_0)} \]
hold, where the last equality follows from the fact that $x_0$ is $f$-closed.  Combining these identities, we conclude that  
\[\dim [x_0,p]=\dim [f(x_0),f(p)] + \dim G_p, \] as desired.
\end{proof}

\begin{rem} {\rm Let us interpret $X$ as an $\FF_1$-scheme and the closed subsets $C_p$ as the ``subvarieties'' of $X$. For a map $f\colon X\to Y$, the surjectivity of $f\colon C_p\to C_{f(p)}$ is equivalent to being ``dominant''  (i.e., density of the image), which corresponds to a dominant morphism of ``varieties''. Under this dictionary, the identity 
$$\dim C_p=\dim C_{f(p)}+\dim G_p$$ agrees with the classical dimension formula for a dominant morphism of algebraic varieties (note that $G_p$ is the generic fibre of $f\colon C_p\to C_{f(p)}$).}
\end{rem}

To conclude this section, we examine what the properness of a space $X$ translates to in the dual space $X^{\rm op}$.  This result was unexpected,  as it does not possess a known analog in classical geometric frameworks.

\begin{thm}\label{properversusCM} Let $X$ be a locally dualizable space. Then, $X$ is proper if and only if $X^{\text{\rm op}}$ is Cohen--Macaulay and its canonical sheaf is locally trivial of rank 1; that is, $$\omega_{U_{\wh p} }\simeq \ZZ$$ for every $\wh p \in X^{\text{\rm op}}$.
\end{thm}

\begin{proof} Assume that $X$ is proper. For each $\wh p\in X^{\text{\rm op}}$, observing that  $U_{\wh p}^*=(\partial C_p)^{\text{\rm op}}$, we have:
\[ \RR\Gamma_{\text{\rm red}}(U_{\wh p}^*,\ZZ) \overset{\ref{X=Xop}} = \RR\Gamma_{\text{\rm red}}(\partial C_p,\ZZ) \overset{\ref{proper+dualizable} }=\ZZ[-\dim  \partial C_p] = \ZZ[-\dim U_{\wh p}^*],\] which implies that  $U_{\wh p}^*$ is a cohomological sphere.
The claim follows directly from  Corollary \ref{Gor}.

Conversely, assume that $X^{\text{\rm op}}$ is Cohen-Macaualy and satisfies $\omega_{U_{\wh p} }\simeq \ZZ$ for every $\wh p \in X^{\text{\rm op}}$. For each $p\in X$, the space $\partial C_p$ is pure  because $(\partial C_p)^{\text{\rm op}}=U_{\wh p}^*$ and $U_{\wh p}$ is pure (since it is Cohen-Macaulay, see \cite[Thm. 5.18]{ST2}). Moreover, $\partial C_p$ is a cohomological sphere because
$ \RR\Gamma_{\text{\rm red}}(\partial C_p,\ZZ)     =\RR\Gamma_{\text{\rm red}}(U_{\wh p}^*,\ZZ)$ and $U_{\wh p}^*$ is a cohomological sphere by Corollary \ref{Gor}. By Theorem \ref{proper+dualizable}, $X$ is proper.
\end{proof}

\section {Examples}\label{section-examples}

In this section we present an elementary study of properness on projective  spaces, spheres and 1-dimensional finite spaces, which we refer to as $\mathbb{F}_1$-curves.

\begin{prop} The projective space $\PP^n_{\FF_1}$ is proper. Consequently, every projective simplicial complex is proper and every projective morphism is proper. The product of two projective simplicial complexes is proper (though not projective in general). For any finite space $X$, its barycentric subdivision $\beta X$ is proper.
\end{prop}

\begin{proof} The properness of  $\PP^n_{\FF_1}$ follows directly from Corollary \ref{dualizante=canonico}, as  its dualizing complex is given by $\ZZ_{\{\un\}}[n]$ and $\ZZ_{\{\un\}}$ is a canonical complex. Alternatively, one can apply Theorem \ref{proper+dualizable}, noting that for every $p\in\PP^n_{\FF_1}$, the closure $C_p$ is a projective space.
The remaining consequences are immediate.
\end{proof}

Let us identify $\PP^n_{\FF_1}$ (resp. $\PP^m_{\FF_1}$) with the poset of non-empty subsets of a set $\Delta$ with $n+1$ elements (resp.,  a set $\Delta'$ with $m+1$ elements). A map $\phi\colon \Delta\to\Delta'$ induces a continuous map
\[\PP(\phi)\colon \PP^n_{\FF_1}\to \PP^m_{\FF_1}\] which maps a subset $S$ of $\Delta$ to the subset $\phi(S)$ of $\Delta'$. We say that $\PP(\phi)$ is the \emph{simplicial map} induced by $\phi$. If $\phi$ is injective, then $\PP(\phi)$ is a closed immersion. If $\phi$ is surjectivo, then $\PP(\phi)$ is also surjective. 
\begin{prop}\label{linear-proper} The simplicial map $\PP(\phi)$ is proper.
\end{prop}
\begin{proof} By Corollary \ref{c-proper=proper} it suffices to show that $\PP(\phi)$ is c-proper. For every $p\in \PP^n_{\FF_1}$ (resp. $q\in \PP^m_{\FF_1}$), let  $\Delta_p$ (resp.  $\Delta_q\subseteq \Delta'$) denote the corresponding subset of $\Delta$. The restriction $\PP(\phi)_{\vert C_p}\colon C_p\to C_{\PP(\phi)(p)}$ coincides with $\PP(\phi_{\vert \Delta_p})$, for the map $\phi_{\vert \Delta_p}\colon \Delta_p\to \Delta_{\PP(\phi)(p)}=\phi(\Delta_p)$, which is a surjective simplical map. To conclude, it is enough to verify that for every surjective simplical map  $\PP(\phi)$, the fibres $\PP(\phi)^{-1}(q)$ are cohomologically trivial. Observe that
\[ \PP(\phi)^{-1}(q)=\left\{ p\in \PP^n_{\FF_1}: \,\aligned  p\cap \phi^{-1}(j)&=\emptyset,\text{ for } j\notin q
\\ p\cap \phi^{-1}(j)&\neq\emptyset,\text{ for } j\in q\endaligned\right\}.\]
This yields a natural homeomorphic identification  $\PP(\phi)^{-1}(q)\simeq \underset{j\in q}\prod \PP^{n_j}_{\FF_1}$, where $n_j=\# \phi^{-1}(j) -1$. Since products of projective spaces are cohomologically trivial, the assertion holds.
\end{proof}

\begin{defn}{\rm Let $K,L$ be two projective simplicial complexes. Thus, $K$ (resp. $L$) is a closed subset of $\PP^n_{\FF_1}$ (resp.  $\PP^m_{\FF_1}$), where $n+1$ is the number of vertices of $K$ (resp., $m+1$ the number of vertices of $L$). A {\em simplical map} $f\colon K\to L$ is a continuous map that extends to a simplicial map $\PP(\phi)\colon \PP^n_{\FF_1}\to \PP^m_{\FF_1}$, meaning that there exists a commutative diagram:
$$
\xymatrix{ \PP^n_{\FF_1} \ar[r]^{\PP(\phi)} & \PP^m_{\FF_1} \\ K \ar@{^(->}[u]  \ar[r]^{f} \ \   & \  \ L . \ar@{^(->}[u]   }  
$$ }
\end{defn} From Proposition \ref{linear-proper} we immediately obtain:
\begin{cor}\label{simplicialisproper} Every simplicial map $f\colon K\to L$   is proper.
\end{cor} 

As an example of this result, for every continuous map $f\colon X\to Y$, the induced map between the barycentric subdivisions  $\beta(f)\colon \beta X\to\beta Y$ is a simplicial map, hence $\beta(f)$ is a proper map.

The converse of Corollary \ref{simplicialisproper} also holds, yielding the following characterization:

\begin{thm}\label{simplicial=proper} Let $f\colon K\to L$ be a continous map between projective simplicial complexes. The following conditions are equivalent:
\begin{enumerate}
\item $f$ is $c$-proper.
\item $f$ is proper.
\item $f$ is simplicial.
\end{enumerate} 
\end{thm}

\begin{proof} A map $f$ satisfies any of these conditions if and only if its restriction $f_{\vert C_p}\colon C_p\to C_{f(p)}$ does for every $p\in K$. Thus, we are reduced to proving the theorem for a surjective continuous map $f\colon \PP^n_{\FF_1}\to \PP^m_{\FF_1}$. In this setting (1) and (2) are equivalent by Corollary \ref{c-proper=proper}, and we already know that (3) implies (2). 

To conclude, we show that if $f$ is proper, then $f$ must be simplicial. We proceed by induction on $n$, the case $n=0$ being trivial. We may assume that $m<n$, since $f$ is an isomorphism when $m=n$. Let $H_0,\dots,H_n$ be the standard hyperplanes of $\PP^n_{\FF_1}$. By the induction hypothesis, the restriction of $f$ to each $H_i$ is simplicial; thus,  the restriction of $f$ to $H_0\cup\dots\cup H_n=\PP^n_{\FF_1}-\{\un\}$ is simplicial. This provides a simplicial map $\PP(\phi)\colon \PP^n_{\FF_1}\to \PP^m_{\FF_1}$ that coincides with $f$ away from the generic point $\mathbf{1}$, and it remains to verify that $\PP(\phi)(\un)=\un$. Since $f$ is proper, Corollary \ref{properonvarieties} ensures that  $\dim \PP^n_{\FF_1}=\dim \PP^n_{\FF_1}+\dim G_\un$, where  $G_\un=f^{-1}(\un)$. This forces  $\dim f^{-1}(\un)=n-m>0$; hence there exists a non-generic point $x\in \PP^n_{\FF_1}$,   such that $f(x)=\un$. It follows that $\PP(\phi)(x)=f(x)=\un$, which implies $\PP(\phi)(\un)=\un$, completing the induction step.
 \end{proof}
 
 \begin{rem}\label{simplicial=proper2} {\rm This theorem remains valid for simplicial complexes instead of projective simplicial complexes. A map $\phi\colon \Delta\to\Delta' $ between sets with $n$ and $m$ elements induces a continuous map $\A(\phi)\colon \A^n_{\FF_1}\to \A^m_{\FF_1} $, which is called the simplicial map associated with $\phi$. A simplicial map $f\colon K\to L$ between simplicial complexes is a map that extends to a simplicial map $\A(\phi)$. The same arguments as above show that a map $f\colon K\to L$ is simplicial if and only if it is c-proper, which in turn is equivalent to being proper. }
 \end{rem}

\begin{prop} The $n$-dimensional sphere  ${\mathbb S}^n$ is proper. Hence every closed subset of a product of spheres is proper. Furthermore, a continuous map $f\colon {\mathbb S}^n\to {\mathbb S}^m$ is proper if and only if it is c-proper.
\end{prop}

\begin{proof} For every $p\in {\mathbb S}^n$, one has $\partial C_p= {\mathbb S}^{\dim \partial C_p}$. By Theorem \ref{proper+dualizable}, ${\mathbb S}^n$ is proper. If $f$ is c-proper,   its properness follows directly from Corollary \ref{c-proper=proper}.

\end{proof}

\begin{prop} The affine space $\AAA^n_{\FF_1}$ is nor proper  for $n>0$. More generally, every local  space of dimension $n>0$ is not proper.
\end{prop}

\begin{proof} If $p\in X$ is a non-closed point with $\dim C_p=1$, then $\partial C_p=\{\0\}$,  which implies  $\partial C_p-\{\0\}=\emptyset$. Since the empty set is not cohomologically trivial,  Theorem \ref{thmproperspace}   ensures that  $X$ is not proper.
\end{proof}

\medskip
\noindent{\S\it \ \ Complete curves over $\FF_1$.}\medskip

\begin{defn}{\rm  An {\em $\FF_1$-curve} $\C$ is   a  $1$-dimensional finite space. If $\C$ is proper, we say that $\C$ is a {\em complete} $\FF_1$-curve.}
\end{defn}

\begin{prop} Every $\FF_1$-curve is locally simplicial. In particular, it is locally dualizable.
\end{prop}

\begin{proof} Let $p\in \C$ and let us check that $U_p$ is simplicial. 
If $p$ is a generic point, then $U_p=\{p\}$, and the result is immediate. If $p$ is not generic and $r$ denotes  the number of generic points of $U_p$, then $U_p$ is isomorphic to the union of the coordinate axes   of $\A^r_{\FF_1}$, which constitutes  a closed subset of $\A^r_{\FF_1}$. 
\end{proof}

Note that every $\FF_1$-curve is Cohen--Macaulay.

\begin{prop}\label{completecurves} Let $\C$ be an irreducible $\FF_1$-curve. Then $\C$ is complete if and only if $\C=\PP^1_{\FF_1}$. That is, $\PP_{\FF_1}^1$ is the unique complete and irreducible $\FF_1$-curve. Consequently, a connected $\FF_1$-curve $\C$ is complete if and only if every irreducible component of $\C$ is isomorphic to $\PP^1_{\FF_1}$. Moreover, a continuous map $f\colon \C\to \C'$ between complete  $\FF_1$-curves is proper if and only if it is closed.
\end{prop}

\begin{proof} Let $g$ be the unique generic point of $\C$. We must show that $\C$ has exactly two closed points. This follows immediately since  $\C-\{g\}=\partial \C$ is a 0-dimensional cohomological sphere.

Now, assume that $f\colon \C\to \C'$ is closed, and let us prove that the (surjective) map $f_{\vert C_g}\colon C_g\to C_{f(g)}$ is proper for every generic point $g\in\C$.  If $g$ is a closed point, then $C_g=\{g\}$  $f_{\vert C_g}$ is trivially proper. If $g$ is not closed, then $C_g=\PP^1_{\FF_1}$. If $f(g)$ is a closed point, then $C_{f(g)}$ is a single point, so $f_{\vert C_g}$ is proper because $\PP^1_{\FF_1}$ is proper. If $f(g)$ is not closed, then $ C_{f(g)}=\PP^1_{\FF_1}$ and $f_{\vert C_g}$ is an isomorphism.
\end{proof}

\begin{rem}\label{completecurves2}{\rm (1) Let $\C$ be a complete $\FF_1$-curve, let $E$ be the set of non-closed generic points of $\C$, and  let $V$ be the set of closed points. This data naturally defines an undirected multigraph without loops $(V,E)$, where each edge $g \in E$ has as its endpoints the vertices in $\partial C_g$.  Conversely, every undirected multigraph without loops uniquely determines a complete $\mathbb{F}_1$-curve.

(2) The circle ${\mathbb S}^1$ constitutes the most basic example of a complete $\FF_1$-curve that is not projective.}
\end{rem}

\end{document}